\documentclass[amstex,12pt,reqno]{amsart}
\usepackage{amsmath, amssymb, amsthm}
\usepackage{mathrsfs}

\usepackage{mathtools}
\mathtoolsset{showonlyrefs=true}
\usepackage{color}

\theoremstyle{plain}
\newtheorem{theorem}{Theorem}[section]
\newtheorem{proposition}[theorem]{Proposition}
\newtheorem{lemma}[theorem]{Lemma}
\newtheorem{corollary}[theorem]{Corollary}

\theoremstyle{definition}
\newtheorem{remark}{Remark}[section]
\newtheorem{definition}{Definition}[section]

\title[Simultaneous uniformization of chord-arc curves]
{Simultaneous uniformization of chord-arc curves \\and BMO Teich\-m\"ul\-ler space}
\author[K. Matsuzaki]{Katsuhiko Matsuzaki}
\address{Department of Mathematics, School of Education, Waseda University \endgraf
Shinjuku, Tokyo 169-8050, Japan}
\email{matsuzak@waseda.jp}



\thanks{Research supported by supported by 
Japan Society for the Promotion of Science (KAKENHI 23K25775 and 23K17656).}

\begin{document}

\maketitle

\begin{abstract}
This article surveys and develops the use of simultaneous uniformization
for the study of chord-arc curves in the BMO Teich\-m\"ul\-ler space.
The method of simultaneous uniformization provides a unified
complex-analytic framework in which chord-arc curves are parametrized by
their BMO embeddings and the logarithm of derivatives of these embeddings
form a biholomorphic image in the Banach space of BMO functions. 
We review this correspondence and its consequences, such as the relation
to reparametrizations by strongly quasisymmetric homeomorphisms, in a rather
self-contained manner in order to highlight the coherence of the
approach. 

The main new contribution of this exposition concerns the Cauchy
transform of BMO functions on a chord-arc curve. We show that the Cauchy
transform is expressed through the derivative of the biholomorphic map
arising from simultaneous uniformization, and consequently depends
holomorphically on the variation of the chord-arc curve. This result
connects classical singular integral operators with the complex
structure of Teich\-m\"ul\-ler spaces and illustrates the effectiveness
of the method. We also outline the parallel theory in the VMO
Teich\-m\"ul\-ler space, which exhibits further structural properties.
\end{abstract}

\section{Introduction}\label{section1}
The Teichm\"uller theory was originally developed to parametrize
deformations of complex structures on surfaces. It has played a central
role in the study of the geometry of hyperbolic 2- and 3-manifolds.
Since such deformations are realized by quasiconformal mappings,
Teichm\"uller theory naturally incorporates these mappings as
fundamental tools. The Teichm\"uller space itself carries both complex
and metric structures, and the interaction between the geometry of the
phase space and that of the parameter space is a distinctive feature of
the theory. Extending this viewpoint, Teichm\"uller theory can also be
applied to parameter spaces of geometric objects arising in the study of
function spaces.

The universal Teich\-m\"ul\-ler space is a space that encompasses all Teich\-m\"ul\-ler spaces formulated by quasiconformal mappings. It can be viewed as the space of all normalized quasisymmetric homeomorphisms on the real line. Depending on the regularity of these mappings, the subspaces contained within the universal Teich\-m\"ul\-ler space split into two directions. The Teich\-m\"ul\-ler space of a compact hyperbolic surface is typically the space of all totally singular quasisymmetric homeomorphisms, while many Teich\-m\"ul\-ler spaces associated with function spaces consist of (locally) absolutely continuous quasisymmetric homeomorphisms.

These subspaces $T_X$ of the universal Teich\-m\"ul\-ler space $T$ can be defined by various means, such as by quasisymmetric homeomorphisms, by Beltrami coefficients as the complex dilatations of quasiconformal mappings, by the (pre-)Schwarzian derivatives of quasiconformally extendable conformal mappings on the half-plane $\mathbb H$, or by quasicircles which are the images of the real line $\mathbb R$ under quasiconformal self-homeomorphisms of the whole plane $\mathbb{C}$. The first step in studying the structure of a given Teich\-m\"ul\-ler space $T_X$ is to establish the correspondence between these different representations. Furthermore, if $T_X$ is defined by a family of absolutely continuous quasisymmetric homeomorphisms $h$, the logarithm of their derivatives $\log h'$ form a function space, which can be used to provide $T_X$ with more interesting analytic structures. We call generically such $T_X$ absolutely continuous Teich\-m\"ul\-ler spaces.

To investigate a family of quasicircles, we introduce a method called simultaneous uniformization by Bers, which has been used in the theory of deformation of quasi-Fuchsian groups and associated hyperbolic 3-manifolds. 
Especially, complex analytic aspects of Thurston's theory are built upon this foundation. 
We apply this method for the analysis of families of curves $\Gamma$. This allows us to coordinate them in the direct product 
$T_X^+ \times T_X^-$ of their corresponding Teich\-m\"ul\-ler spaces $T_X$. Furthermore, 
we can prove the existence of a biholomorphic homeomorphism of $T_X^+ \times T_X^-$ onto a domain 
in the function space to which $\log \gamma'$ of a quasisymmetric embedding $\gamma: \mathbb{R} \to \Gamma \subset \mathbb{C}$ belongs.


Specifically, we focus on the BMO Teich\-m\"ul\-ler space $T_B$ in this
paper, which is an active and rapidly developing area of research with
deep connections to real analysis. The space $T_B$ also provides a
natural setting for many absolutely continuous Teich\-m\"ul\-ler spaces
$T_X$, since in numerous cases these spaces embed into $T_B$. Within
this framework, we pay particular attention to the subregion defined by
rectifiable quasicircles known as chord-arc curves, and study structural
problems for this family of curves from the viewpoint of
Teich\-m\"ul\-ler theory. Chord-arc curves represent a comparatively weak
regularity class among non-fractal curves, and they have been the focus
of extensive investigations in real analysis concerning associated
function spaces.

A locally rectifiable Jordan curve passing through infinity is called a chord-arc curve if the length of any ``arc" along the curve between any two points is uniformly bounded by a constant multiple of the length of the line segment ``chord" connecting the two points. If we replace the length of the ``arc" with its diameter, it characterizes a quasicircle. A quasicircle $\Gamma$ passing through infinity is also characterized as the image of $\mathbb{R}$ under a quasiconformal self-homeomorphism of $\mathbb{C}$. 
The set of all such quasicircles, up to affine translations, can be identified with the universal Teich\-m\"ul\-ler space $T$. Similarly, chord-arc curves are the images of $\mathbb{R}$ under bi-Lipschitz self-homeomorphisms of $\mathbb{C}$.

A quasisymmetric homeomorphism $h: \mathbb{R} \to \mathbb{R}$ is the extension of a quasiconformal self-homeomorphism of $\mathbb{H}$. All the normalized quasisymmetric self-homeomorphisms form a group denoted by $\rm QS$ and it can be identified with the universal Teich\-m\"ul\-ler space $T$. If $h$ is locally absolutely continuous and its derivative $h'$ belongs to the class of Muckenhoupt's $A_\infty$-weights, then $h$ is called strongly quasisymmetric. The set of all normalized strongly quasisymmetric self-homeomorphisms is a subgroup of 
$\rm QS$ denoted by $\rm SQS$.

BMO functions are an important class of functions that appear and play an essential role in many problems of real analysis. Let ${\rm BMO}(\mathbb{R})$ be the space of all complex-valued BMO functions on $\mathbb{R}$, and let ${\rm BMO}^*(\mathbb{R})$ be the subset consisting of $\phi \in {\rm BMO}(\mathbb{R})$ for which $|e^\phi|$ is an $A_\infty$-weight. This forms a convex open subset of the complex Banach space ${\rm BMO}(\mathbb{R})$. For any $h \in {\rm SQS}$, it can be understood that $\log h'$ is in ${\rm Re}\,{\rm BMO}^*(\mathbb{R})$, the subset consisting of real-valued functions. In fact, the correspondence ${\rm SQS} \to {\rm Re}\,{\rm BMO}^*(\mathbb{R})$ is bijective.

The BMO Teich\-m\"ul\-ler space $T_B$ in $T$,
introduced by Astala and Zinsmeister \cite{AZ}, possesses several characterizations as described above. As a space of quasisymmetric homeomorphisms, it coincides with ${\rm SQS}$. The corresponding space of Beltrami coefficients on the half-plane $\mathbb{H}$ is denoted by $M_B(\mathbb{H})$, and $T_B$ is defined by its quotient under the Teich\-m\"ul\-ler equivalence. More explicitly, an element in $M_B(\mathbb{H})$ is defined by the Carleson measure condition for a Beltrami coefficient.

We can apply the method of simultaneous uniformization to chord-arc curves. For any $\mu^+ \in M_B(\mathbb{H}^+)$ and $\mu^- \in M_B(\mathbb{H}^-)$, where $\mathbb{H}^\pm$ are the upper and lower half-planes, we consider the normalized quasiconformal self-homeomorphism $G(\mu^+,\mu^-)$ having the prescribed complex dilatations on $\mathbb{H}^\pm$. A BMO embedding $\gamma:\mathbb{R} \to \mathbb{C}$ is given by $\gamma=G(\mu^+,\mu^-)|_{\mathbb{R}}$ and it is well-defined by a pair of Teich\-m\"ul\-ler classes $([\mu^+],[\mu^-])$. The space of all BMO embeddings is identified with the product $T_B^+ \times T_B^-$ of the BMO Teich\-m\"ul\-ler spaces. We call these pairs the Bers coordinates. Among the Bers coordinates of BMO embeddings $\gamma:\mathbb{R} \to \mathbb{C}$, the subset whose image $\Gamma=\gamma(\mathbb{R})$ is a chord-arc curve is defined to be $\widetilde{T}_C \subset T_B^+ \times T_B^-$. Any $\gamma=\gamma([\mu^+],[\mu^-])$ for $([\mu^+],[\mu^-]) \in \widetilde{T}_C$ is locally absolutely continuous, and $\log \gamma'$ belongs to ${\rm BMO}^*(\mathbb{R})$.

The following theorem is a fundamental assertion in studying chord-arc curves from the viewpoint of Teich\-m\"ul\-ler spaces. This has been proved in Wei and Matsuzaki (abbreviated as WM hereafter) \cite{WM-0}, and in this paper, we review its proof with a view towards studying the Cauchy transform on a chord-arc curve. Additionally, along the way of its proof, we include several different arguments from the existing ones.

\begin{theorem}\label{fundamental}
$\widetilde{T}_C$ is an open subset of $T_B^+ \times T_B^-$, and the map $\Lambda([\mu^+],[\mu^-])=\log \gamma'$ defined for $\gamma=\gamma([\mu^+],[\mu^-])$ and $([\mu^+],[\mu^-]) \in \widetilde{T}_C$ is a biholomorphic homeomorphism onto an open subset of ${\rm BMO}^*(\mathbb{R})$ containing ${\rm Re}\,{\rm BMO}^*(\mathbb{R})$.
\end{theorem}

The importance of this result lies in the fact that the holomorphic dependence of BMO functions can be transformed into a complex-analytic structure of the Teich\-m\"ul\-ler space. 
Chord-arc curves have been studied from the perspective of harmonic analysis as plane curves, but research from Teich\-m\"ul\-ler space theory through such simultaneous uniformization has not been done so far. Hereinafter, we will see that such a perspective can provide concise proofs for some discussions related to chord-arc curves and clarify them.

In the above theorem, we also see that the derivative $d_{([\mu^+],[\mu^-])}\Lambda$ of the biholomorphic homeomorphism 
$\Lambda$ at $([\mu^+],[\mu^-]) \in \widetilde{T}_C$ induces an isomorphism between the tangent spaces. 
Hence, according to the direct sum decomposition of the tangent space associated with $T_B^+ \times T_B^-$, 
$d_{([\mu^+],[\mu^-])}\Lambda$ induces the topological direct sum decomposition of ${\rm BMO}(\mathbb{R})$ 
and the bounded projections to these factors. We denote these projections in ${\rm BMO}(\mathbb{R})$ by 
$P_{([\mu^+],[\mu^-])}^+$ and $P_{([\mu^+],[\mu^-])}^-$. It holds $P_{([\mu^+],[\mu^-])}^+\phi+P_{([\mu^+],[\mu^-])}^-\phi=\phi$
for every $\phi \in {\rm BMO}(\mathbb{R})$. Considering these properties of the derivative of $\Gamma$ is a new
observation in this paper.

Let $\Gamma$ be a chord-arc curve, and $\Omega^+$ and $\Omega^-$ the complementary domains of $\mathbb{C}$ divided by $\Gamma$. The Cauchy integrals of a BMO function $\psi$ on $\Gamma$ are defined by
$$
({P}^\pm_\Gamma \psi)(\zeta)=\frac{-1}{2\pi i} \int_{\Gamma} 
\left(\frac{\psi(z)}{\zeta-z}-\frac{\psi(z)}{\zeta_0^\pm-z}\right)dz
\qquad (\zeta \in \Omega^\pm)
$$
for some fixed $\zeta_0^\pm \in \Omega^\pm$, where the line integrals over $\Gamma$ 
are taken in the positive directions with respect to $\Omega^\pm$. 
These are holomorphic functions on $\Omega^\pm$, and 
have non-tangential limits almost everywhere on $\Gamma$. These boundary functions on $\Gamma$ 
are called the Cauchy projections of $\psi$ on $\Gamma$, denoted 
by the same notations ${P}^\pm_\Gamma \psi$. The Plemelj formula implies that ${P}^+_\Gamma \psi+{P}^-_\Gamma \psi=\psi$.

We compare the Cauchy projections ${P}^\pm_\Gamma$ with the conjugates of the projections $P_{([\mu^+],[\mu^-])}^\pm$ 
under $\gamma$.
Applied to a BMO function $\psi$ on $\Gamma$, they induce the holomorphic functions on $\Omega^\pm$ having the same jump $\psi$
(or the same sum $\psi$ depending on the choice of directions of the line integral) across $\Gamma$.
In these circumstances, general arguments deduce that they coincide with each other. Thus, the boundedness on
the BMO function space and the holomorphic dependence on $\gamma$ possessed by $P_{([\mu^+],[\mu^-])}^\pm$ 
can be imported into the properties of
the Cauchy projections ${P}^\pm_\Gamma$.

\begin{theorem}\label{intro2}
Let $\gamma=\gamma([\mu^+],[\mu^-])$ be a BMO embedding for $([\mu^+],[\mu^-]) \in \widetilde{T}_C$ 
with a chord-arc curve $\Gamma=\gamma(\mathbb{R})$ as its image. 
Then, the Cauchy projections ${P}^\pm_\Gamma$ are the conjugates by $\gamma$ of the projections $P_{([\mu^+],[\mu^-])}^\pm$ 
associated with the topological direct sum decomposition of ${\rm BMO}(\mathbb{R})$. Moreover, $P_{([\mu^+],[\mu^-])}^\pm$ 
depend holomorphically on $([\mu^+],[\mu^-]) \in \widetilde{T}_C$ as bounded linear operators acting on ${\rm BMO}(\mathbb{R})$.
\end{theorem}

The Cauchy transform of a BMO function on a chord-arc curve $\Gamma$ is defined by the singular integral
$$
({\mathcal H}_\Gamma \psi)(\xi)={\rm p.v.} \frac{1}{\pi} \int_{\Gamma} 
\left(\frac{\psi(z)}{\xi-z}-\frac{\psi(z)}{\zeta_0^\pm-z}\right)dz
\quad(\xi \in \Gamma).
$$
By the Plemelj formula, ${\mathcal H}_\Gamma$ can be represented by the Cauchy projections ${P}^\pm_\Gamma$. The Cauchy transform of BMO functions on a chord-arc curve is an important subject in real analysis. 
This originates in Calder\'on's work.
Simultaneous uniformization makes it possible to investigate it in the framework of complex-analytic Teich\-m\"ul\-ler space theory.
The following corollary to Theorem \ref{intro2} highlights a new claim of this paper, 
demonstrating how the method of simultaneous uniformization yields the holomorphic dependence of the Cauchy transform.

\begin{corollary}
The conjugate of the Cauchy transform ${\mathcal H}_\Gamma$ under $\gamma$ represented by
$$
{\mathcal H}_{([\mu^+],[\mu^-])}=-i( P^+_{([\mu^+],[\mu^-])} -P^-_{([\mu^+],[\mu^-])})
$$
depend holomorphically on $([\mu^+],[\mu^-]) \in \widetilde{T}_C$
as the bounded linear operators on ${\rm BMO}(\mathbb R)$. 
\end{corollary}

This result should be useful for considering and simplifying several problems on function spaces on chord-arc curves.
We demonstrate an application.

Finally, we touch on the theory of the little subspace of $T_B$, the VMO Teich\-m\"ul\-ler space $T_V$. This corresponds to the closed subspace ${\rm VMO}(\mathbb{R})$ of ${\rm BMO}(\mathbb{R})$, which consists of BMO functions whose norm equalities satisfy a canonical vanishing property. Results for various little subspaces defined by their vanishing conditions are often stated in a way that mirrors those for the original Teich\-m\"ul\-ler spaces. However, in the case of $T_V$, it possesses preferable properties that $T_B$ does not, such as a topological group structure and a global section for the Teich\-m\"ul\-ler projection. Moreover, subtle differences arise when considering $T_V$ defined on $\mathbb{S}$ versus on $\mathbb{R}$. Emphasizing these points, we provide a concise exposition of this theory and present new observations as well.

In this article we restrict our attention to the BMO Teich\-m\"ul\-ler
space and chord-arc curves. However, the method of simultaneous
uniformization also applies effectively to another important class,
namely the integrable Teich\-m\"ul\-ler space and the associated
Weil--Petersson curves. Parallel arguments in that setting lead to new
results, which are presented in a separate paper \cite{Ma}. Taken
together, these developments indicate that the framework discussed here
provides a broader perspective on absolutely continuous
Teich\-m\"ul\-ler spaces, with the present exposition focusing on the
BMO case while outlining its connections to related theories.

\medskip
\noindent
{\bf Acknowledgements.}
The author is grateful to the referees for their careful reading of the manuscript and for their many valuable comments, which have helped to improve both the clarity and the presentation of this work. In particular, the present version of the proof of Proposition \ref{convex2} follows a simplification suggested by one of the referees.

\section{BMO Teich\-m\"ul\-ler space and BMOA}\label{section2}

Let $M(\mathbb H)$ denote the open unit ball of the Banach space $L^{\infty}(\mathbb H)$
of all essentially bounded measurable functions on the half-plane $\mathbb H$. 
An element in $M(\mathbb H)$ is called a
{\it Belt\-rami coefficient}.
The {\it universal Teich\-m\"ul\-ler space} $T$ is the set of all Teich\-m\"ul\-ler equivalence classes $[\mu]$ of
Beltrami coefficients $\mu$ in $M(\mathbb H)$. Here, $\mu_1$ and $\mu_2$ in $M(\mathbb H)$ are equivalent if
$h(\mu_1)=h(\mu_2)$ on $\mathbb R$, where $h(\mu)=H(\mu)|_{\mathbb R}$ 
is the boundary extension of the quasiconformal self-homeomorphism $H=H(\mu)$
of $\mathbb H$, called a {\it quasisymmetric homeomorphism}, such that its complex dilatation $\bar \partial H/\partial H$ is 
$\mu \in M(\mathbb H)$ and $h(\mu)$ satisfies the normalization condition
keeping the points $0$, $1$ and $\infty$ fixed.

We denote the quotient projection by
$\pi:M(\mathbb H) \to T$, which is called the {\it Teich\-m\"ul\-ler projection}.
Thus, we can identify $T$ with the set $\rm QS$ of all normalized quasisymmetric homeomorphisms $h(\mu)$ for $\mu \in M(\mathbb H)$.
The topology of $T$ is defined as the quotient topology induced from $M(\mathbb H)$ by $\pi$.
The universal Teich\-m\"ul\-ler space possesses the group structure under the identification
$T \cong \rm QS$.
The composition $h(\mu) \circ h(\nu)$ in $\rm QS$ is denoted by $[\mu]\ast [\nu]$ in $T$ and
the inverse $h(\mu)^{-1}$ is denoted by $[\mu]^{-1}$.
For every $[\nu] \in T$, the right translation $r_{[\nu]}:T \to T$ on the group $T$
is defined by $[\mu] \mapsto [\mu] \ast [\nu]$.

Let $F^\mu$ denote the normalized ($0$, $1$ and $\infty$ are fixed) 
quasiconformal self-homeomorphism of $\mathbb C$ whose complex dilatation
is $\mu \in M(\mathbb H^+)$ on the upper half-plane $\mathbb H^+$ and $0$ on the lower half-plane $\mathbb H^-$.
For $\mu_1$ and $\mu_2$ in $M(\mathbb H^+)$,
we see that $\pi(\mu_1)=\pi(\mu_2)$ if and only if $F^{\mu_1}|_{\mathbb H^-}=F^{\mu_2}|_{\mathbb H^-}$. 

We define the following spaces of holomorphic functions $\Psi$ and $\Phi$ on $\mathbb{H}$ as follows:
\begin{align*}
A(\mathbb{H}) &= \{\Psi \mid \Vert \Psi \Vert_A = \sup_{z \in \mathbb{H}} |{\rm Im}\,z|^2|\Psi(z)| < \infty \}; \\
B(\mathbb{H}) &= \{\Phi \mid \Vert \Phi \Vert_B = \sup_{z \in \mathbb{H}} |{\rm Im}\,z||\Phi'(z)| < \infty \}.
\end{align*}
Here, $A(\mathbb{H})$ is a complex Banach space with the hyperbolic $L^\infty$-norm $\Vert \cdot \Vert_A$,
and $B(\mathbb{H})$ is the {\it Bloch space} with the semi-norm $\Vert \cdot \Vert_B$. 
By ignoring the difference in constant functions, 
we regard $B(\mathbb{H})$ as a complex Banach space with the norm $\Vert \cdot \Vert_B$.

The {\it pre-Schwarzian derivative map} $L:M(\mathbb H^+) \to B(\mathbb H^-)$
is defined by the
correspondence $\mu \mapsto \log (F^\mu|_{\mathbb H^-})'$ and
the Schwarzian derivative map $S:M(\mathbb H^+) \to A(\mathbb H^-)$ is defined by
$S(\mu)=L(\mu)''-(L(\mu)')^2/2$, where $B(\mathbb{H})$ and $A(\mathbb{H})$ serve appropriate spaces as
the targets of $L$ and $S$, respectively.
Let $D(\Phi)=\Phi''-(\Phi')^2/2$ for $\Phi \in B(\mathbb H^-)$, which satisfies
$S=D \circ L$.
Then, $D$ restricted to $L(M(\mathbb H^+))$ is a holomorphic bijection onto $S(M(\mathbb H^+))$.

It is proved that $S$ is 
a holomorphic split submersion onto the bounded contractible domain
$S(M(\mathbb H^+))$ in $A(\mathbb H^-)$. Since $S=D \circ L$, we see that 
$D:L(M(\mathbb H^+)) \to S(M(\mathbb H^+))$ is a biholomorphic homeomorphism and
$L$ is also a holomorphic split submersion onto the bounded contractible domain 
$L(M(\mathbb H^+))$ in $B(\mathbb H^-)$.
Moreover,
these maps induce
well-defined injections $\alpha:T \to A(\mathbb H^-)$ such that
$\alpha \circ \pi=S$ and $\beta:T \to B(\mathbb H^-)$ such that
$\beta \circ \pi=L$. We call $\alpha$ the Bers embedding and $\beta$ the {\it pre-Bers embedding}.
By the facts that $S$ and $L$ are split submersions, we see that $\alpha$ and $\beta$ are homeomorphisms
onto the bounded contractible domains $\alpha(T)=S(M(\mathbb H^+))$ in $A(\mathbb H^-)$ and
$\beta(T)=L(M(\mathbb H^+))$ in $B(\mathbb H^-)$, respectively.

We can refer to a textbook of Lehto \cite{Le} for most of the aforementioned facts on the universal Teichm\"uller space.
Concerning the pre-Schwarzian derivative map, see Sugawa \cite{Su}.

Our subject turns to the BMO Teich\-m\"ul\-ler space $T_B$ introduced by Astala and Zinsmeister \cite{AZ}, which lies in $T$.
In general, 
we say that a measure $\lambda$ on $\mathbb{H}$ is a {\it Carleson measure} if
$$
\Vert \lambda \Vert_c = \sup_{I \subset \mathbb{R}} \frac{\lambda(I \times (0,|I|))}{|I|} < \infty,
$$
where the supremum is taken over all bounded intervals $I$ in $\mathbb{R}$. 
For $\mu \in L^\infty(\mathbb{H})$, we consider an absolutely continuous measure
$\lambda_\mu = |\mu(z)|^2 dxdy/y$ and let
$\Vert \mu \Vert_c = \Vert \lambda_\mu \Vert_c^{1/2}$.
Then, we provide a stronger norm $\Vert \mu \Vert_{\infty} + \Vert \mu \Vert_{c}$ for $\mu$.
Let $L_B(\mathbb H)$ denote the Banach space consisting of all elements $\mu \in L^{\infty}(\mathbb H)$ with 
$\Vert \mu \Vert_{\infty}+\Vert \mu \Vert_{c}<\infty$, namely, $\lambda_\mu$ is a Carleson measure on $\mathbb H$.
Moreover, we define the corresponding space of Beltrami coefficients
as $M_B(\mathbb H) =  M(\mathbb H) \cap L_B(\mathbb H)$.

\begin{definition}
The {\it BMO Teich\-m\"ul\-ler space}
$T_B \subset T$ is defined as $\pi(M_B(\mathbb H))$, equipped with the quotient topology from $M_B(\mathbb H)$ by $\pi$.
\end{definition}

We introduce the space of holomorphic functions
$$
A_B(\mathbb H)=\{\Psi \in A(\mathbb H) \mid \Vert \lambda^{(2)}_\Psi \Vert_c^{1/2}<\infty\},
$$
where $\lambda^{(2)}_\Psi=|\Psi(z)|^2y^3dxdy$ is a Carleson measure on $\mathbb H$. This is a complex Banach space with
norm $\Vert \Psi \Vert_{A_B}=\Vert \lambda^{(2)}_\Psi \Vert_c^{1/2}$. Similarly,
$$
{\rm BMOA}(\mathbb H)=\{\Phi \in B(\mathbb H) \mid \Vert \lambda^{(1)}_\Phi \Vert_c^{1/2}<\infty\},
$$
where $\lambda^{(1)}_\Phi=|\Phi'(z)|^2ydxdy$ is a Carleson measure on $\mathbb H$. 
This space modulo constants is a complex Banach space with
norm $\Vert \Phi \Vert_{\rm BMOA}=\Vert \lambda^{(1)}_\Phi \Vert_c^{1/2}$. 
 
The following result is proved by Shen and Wei \cite[Theorem 5.1]{SWei} extending \cite[Theorem 1]{AZ}. 

\begin{proposition}\label{S-holo}
The Schwarzian derivative map $S$ is a holomorphic map on $M_B(\mathbb H^+)$ into $A_B(\mathbb H^-)$.
Moreover, for each point $\Psi$ in 
$S(M_B(\mathbb H^+))$,
there exists a neighborhood $V_\Psi$ of $\Psi$ in
$A_B(\mathbb H^-)$ and a holomorphic map $\sigma:V_\Psi \to M_B(\mathbb H^+)$ such that
$S \circ \sigma$ is the identity on $V_\Psi$.
\end{proposition}

This fact also leads to a claim on the pre-Schwarzian derivative map $L$ on $M_B(\mathbb H^+)$, and
we have that $L(M_B(\mathbb H^+)) \subset \mathrm{BMOA}(\mathbb H^-)$ and $L:M_B(\mathbb H^+) \to \mathrm{BMOA}(\mathbb H^-)$ is holomorphic.
Moreover,
we also see that $D$ restricted to $L(M_B(\mathbb H^+))$ is a
holomorphic bijection onto $S(M_B(\mathbb H^+))$. 
Hence, likewise to the case of
the universal Teich\-m\"ul\-ler space, the pre-Schwarzian derivative map $L$
satisfies the corresponding property on $M_B(\mathbb H^+)$ to that for $S$ stated in Proposition \ref{S-holo}.
These arguments are given in \cite[Section 6]{SWei}.

\begin{proposition}\label{L-holo}
The pre-Schwarzian derivative map $L$ is a holomorphic map on $M_B(\mathbb H^+)$ into $\mathrm{BMOA}(\mathbb H^-)$,
and at each point in the image $L(M_B(\mathbb H^+))$, there exists a local holomorphic right inverse of $L$.
\end{proposition}

Propositions \ref{S-holo} and \ref{L-holo} also imply that $S(M_B(\mathbb H^+))$ and $L(M_B(\mathbb H^+))$ are open subsets 
in $A_B(\mathbb H^-)$ and ${\rm BMOA}(\mathbb H^-)$ respectively, and 
$D:L(M_B(\mathbb H^+)) \to S(M_B(\mathbb H^+))$ is a biholomorphic homeomorphism.

Under these properties of $S$ and $L$ on $M_B(\mathbb H^+)$,
the Bers embedding $\alpha$ and the pre-Bers embedding $\beta$ of the BMO Teich\-m\"ul\-ler space $T_B$
can be established in the same way as in the case of $T$.
These maps induce complex Banach structures to $T_B$ which are biholomorphically equivalent, and
hence $\alpha$ and $\beta$ are biholomorphic homeomorphisms.

\begin{theorem}\label{Bers}
$(1)$ The Bers embedding $\alpha:T_B \to S(M_B(\mathbb H^+)) \subset A_B(\mathbb H^-)$ is
a homeomorphism onto the image. $(2)$ The pre-Bers embedding $\beta:T_B \to L(M_B(\mathbb H^+)) \subset {\rm BMOA}(\mathbb H^-)$ is
a homeomorphism onto the image.
\end{theorem}

Next, we focus on the relationship between $\rm BMOA$ and $\rm BMO$. 
A locally integrable complex-valued function $\phi$ on $\mathbb R$ is of {\it bounded mean oscillation} (BMO) if
$$
\Vert \phi \Vert_{\mathrm{BMO}}=\sup_{I \subset \mathbb R}\frac{1}{|I|} \int_I |\phi(x)-\phi_I| dx <\infty,
$$
where the supremum is taken over all bounded intervals $I$ on $\mathbb R$ and 
$\phi_I$ denotes the integral mean of $\phi$
over $I$. The set of all complex-valued BMO functions on $\mathbb R$ is denoted by ${\rm BMO}(\mathbb R)$.
This is regarded as a Banach space with norm $\Vert \cdot \Vert_{\mathrm{BMO}}$
by ignoring the difference in complex constant functions.

The {\it John--Nirenberg inequality} for BMO functions (see Garnett \cite[VI.2]{Ga}) 
asserts that
there exist two universal positive constants $C_0$ and $C_{JN}$ such that for any complex-valued BMO function $\phi$, 
any bounded interval $I$ of $\mathbb{R}$, and any $\lambda > 0$, it holds that
\begin{equation}\label{JN}
\frac{1}{|I|} |\{t \in I: |\phi(t) - \phi_I| \geq \lambda \}| \leq 
C_0 \exp\left(\frac{-C_{JN}\lambda}{\Vert \phi \Vert_{\mathrm{BMO}}} \right).
\end{equation}

Concerning the boundary extension of $\Phi \in {\rm BMOA}(\mathbb H)$ to $\mathbb R$, 
we note that $\Phi$ has non-tangential limits 
almost everywhere on $\mathbb R$ and the Poisson integral of this boundary function reproduces $\Phi$.
This links the BMO properties of $\Phi$ on $\mathbb H$ and on $\mathbb R$.
The following theorem is well known, which can be seen from Zhu \cite[Theorems 9.17 and 9.19]{Zhu}.

\begin{theorem}\label{131} 
Let $E(\Phi)$ be the boundary extension of $\Phi \in {\rm BMOA}(\mathbb H)$ defined by
the non-tangential limits on $\mathbb R$.
Then, $E(\Phi) \in {\rm BMO}(\mathbb R)$, and
the trace operator $E:{\rm BMOA}(\mathbb H) \to {\rm BMO}(\mathbb R)$ is 
a Banach isomorphism onto the image.
\end{theorem}

By considering the trace operators $E^+$ and $E^-$
for the half-planes $\mathbb H^+$ and $\mathbb H^-$,
we obtain the closed subspaces 
$E^+(\mathrm{BMOA}(\mathbb H^+))$ and $E^-(\mathrm{BMOA}(\mathbb H^-))$ 
in $\mathrm{BMO}(\mathbb R)$. Functions in
$E^+(\mathrm{BMOA}(\mathbb H^+))$ and $E^-(\mathrm{BMOA}(\mathbb H^-))$ correspond by
complex conjugation.
By the identification
under the Banach isomorphism $E^\pm:\mathrm{BMOA}(\mathbb H^\pm) \to \mathrm{BMO}(\mathbb R)$,
we may regard $\mathrm{BMOA}(\mathbb H^\pm)$ as
closed subspaces of $\mathrm{BMO}(\mathbb R)$. 

Conversely, the projection from $\mathrm{BMO}(\mathbb{R})$ to  
$\mathrm{BMOA}(\mathbb H) \cong E(\mathrm{BMOA}(\mathbb H))$ associated with $E$ is specifically provided
by using the following map.

\begin{definition}
For $\phi \in \mathrm{BMO}(\mathbb R)$,
we define the singular integral
$$
\mathcal{H}(\phi)(x) = \text{p.v.}\frac{1}{\pi}\int_{-\infty}^{\infty} \phi(t) \left(\frac{1}{x-t}+\frac{t}{1+t^2} \right)dt
$$
to be a linear operator on $\mathrm{BMO}(\mathbb{R})$ called the {\it Hilbert transform}. 
\end{definition}

It is well known that $\mathcal H$ gives a Banach automorphism of ${\rm BMO}(\mathbb R)$ 
satisfying ${\mathcal H} \circ {\mathcal H}=-I$
(see \cite[Chapter VI]{Ga}). 
Let
$P^\pm = \frac{1}{2}(I \pm i\mathcal{H})$, which we call the {\it Riesz projections}.
We can apply the Riesz projections $P^\pm$ to $\mathrm{BMO}(\mathbb R)$
as bounded linear operators. We note that $P^++P^-=I$ and $P^+ \circ P^-=P^- \circ P^+=O$
by the definition of $P^\pm$ and the property ${\mathcal H} \circ {\mathcal H}=-I$.
Moreover, 
the images of $P^\pm$ coincide with $E^\pm(\mathrm{BMOA}(\mathbb H^\pm))$, which are the closed subspaces of 
$\mathrm{BMO}(\mathbb R)$ consisting of all elements that extend to holomorphic functions on $\mathbb H^\pm$ by
the Poisson integral. 

\begin{theorem}\label{decomposition}
The Riesz projections $P^\pm$ in $\mathrm{BMO}(\mathbb R)$
are bounded linear projections onto the closed subspaces $E^\pm(\mathrm{BMOA}(\mathbb H^\pm))$.
They yield the topological direct sum decomposition
$$
\mathrm{BMO}(\mathbb R)=E^+(\mathrm{BMOA}(\mathbb H^+)) \oplus E^-(\mathrm{BMOA}(\mathbb H^-)).
$$
\end{theorem}

Holomorphic functions of the upper and the lower half-planes $\mathbb H^\pm$ defined by
the Cauchy integrals of $\phi \in \mathrm{BMO}(\mathbb R)$,
$$
\frac{-1}{2\pi i}\int_{\mathbb R} \phi(t) \left(\frac{1}{z-t}+\frac{t}{1+t^2} \right) dt
\quad (z \in \mathbb{H}^\pm),
$$
are called the {\it Szeg\"o projections} of $\phi$. Here, the integration over $\mathbb R$ is taken in
the increasing direction $\int_{-\infty}^\infty$ when $z \in \mathbb H^+$ and in the decreasing direction $\int_{\infty}^{-\infty}$
when $z \in \mathbb H^-$.

We see that the Szeg\"o projections
give the bounded linear maps $\mathrm{BMO}(\mathbb R) \to \mathrm{BMOA}(\mathbb H^\pm)$ whose composition
with the trace operators
$E^\pm:\mathrm{BMOA}(\mathbb H^\pm) \to \mathrm{BMO}(\mathbb R)$ coincide with the Riesz projections $P^\pm$
(by a special case of the Plemelj formula).
In the sequel, we do not distinguish them, denote both of them by $P^\pm$, and call the Szeg\"o projections.
Moreover, we regard $\mathrm{BMOA}(\mathbb H^\pm)$
as the subspaces of $\mathrm{BMO}(\mathbb R)$
by omitting $E^\pm$ and represent the topological direct sum decomposition of $\mathrm{BMO}(\mathbb R)$ 
in Theorem \ref{decomposition} by 
\begin{equation}\label{omitE}
\mathrm{BMO}(\mathbb R)=\mathrm{BMOA}(\mathbb H^+) \oplus \mathrm{BMOA}(\mathbb H^-).
\end{equation}

\section{Strongly quasisymmetric homeomorphisms}\label{section3}

The universal Teich\-m\"ul\-ler space $T$ is identified with the
set $\rm QS$ of all normalized quasisymmetric homeomorphisms.
A quasisymmetric homeomorphism $h$ can be characterized by the doubling property for
the pull-back of the Lebesgue measure on $\mathbb R$ by $h$.
The BMO Teich\-m\"ul\-ler space $T_B$ is identified with the subset of
quasisymmetric homeomorphisms $h(\mu)=H(\mu)|_{\mathbb R}$ for $\mu \in M_B(\mathbb H)$.
We consider intrinsic characterization of these quasisymmetric homeomorphisms.

\begin{definition}
A quasisymmetric homeomorphism $h:\mathbb R \to \mathbb R$ is called {\it strongly quasisymmetric} if
there are positive constants $K$ and $\alpha$ such that
\begin{equation}\label{alphaK}
\frac{|h(E)|}{|h(I)|}\leq K\left(\frac{|E|}{|I|}\right)^{\alpha}
\end{equation}
for any bounded interval $I \subset \mathbb{R}$ and 
for any measurable subset $E \subset I$.
\end{definition}

We denote the set of all normalized strongly quasisymmetric homeomorphisms by ${\rm SQS}$.
We see that ${\rm SQS}$ is preserved under the composition by \eqref{alphaK}, and
that $h \in {\rm SQS}$ if and only if $h^{-1} \in {\rm SQS}$ by Coifman and Fefferman \cite[Lemma 5]{CF}.
As $\rm QS$ is a group under the composition,
${\rm SQS}$ is a subgroup of $\rm QS$.
We also see that $h \in \rm SQS$ is locally absolutely continuous, and hence
it can be represented as $h(x)=\int_0^x h'(t)dt$. 

\begin{theorem}\label{SQS}
Let $h(\mu)$ be a normalized quasisymmetric homeomorphism of $\mathbb R$. Then, $[\mu]$ belongs to $T_B$ if and only if
$h(\mu)$ is strongly quasisymmetric.
\end{theorem}

The ``only-if'' part of this theorem follows from Fefferman, Kenig and Pipher \cite[Theorem 2.3]{FKP} and 
the ``if'' part follows from \cite[Theorem 4.2]{FKP}. 
Later, we will give a different proof for Theorem \ref{SQS} in Theorems \ref{main1} and \ref{converse}. 
Yet other proofs through other equivalent conditions are summarized in Shen and Wei \cite[Theorem A]{SWei}.

Here, we show a way to extend a strongly quasisymmetric homeomorphism of $\mathbb R$
to a quasiconformal self-homeomorphism of $\mathbb H$  
whose complex dilatation induces a Carleson measure. This is introduced by \cite{FKP} in the course of their arguments
involving the above results.
There is a detailed exposition in WM \cite[Theorem 3.4]{WM-2}. 

Let $\phi(x)=\frac{1}{\sqrt \pi}e^{-x^2}$ and $\psi(x)=\phi'(x)=-2x \phi(x)$.
We extend a strongly quasisymmetric homeomorphism $h:\mathbb R \to \mathbb R$
to $\mathbb H$ by setting a real-analytic diffeomorphism $H: \mathbb{H} \to \mathbb{C}$ by 
\begin{equation*}\label{F}
\begin{split}
&H(x, y) = U(x, y) + iV(x, y);\\
U(x,y)&=(h \ast \phi_y)(x),\ V(x,y)=(h \ast \psi_y)(x),
\end{split}
\end{equation*}
where $\varphi_y(x)=y^{-1} \varphi(y^{-1}x)$ for $x \in \mathbb R$ and $y>0$, and $\ast$ is the convolution.
We call this extension the variant of the {\it Beurling--Ahlfors extension}
by the heat kernel. The original extension given in \cite{BA} uses the kernels
$\phi(x)=\frac{1}{2}1_{[-1,1]}(x)$ and $\psi(x)=\frac{1}{2}1_{[-1,0]}(x)-\frac{1}{2}1_{[0,1]}(x)$.

The quasiconformal extension theorem can be summarized as follows.
The latter statement is in \cite[Proposition 3.2]{WM-3}.

\begin{theorem}\label{FKP}
For $h \in{\rm SQS}$, 
the map $H$ given by
the variant of the Beurling--Ahlfors extension by the heat kernel is a quasiconformal real-analytic
self-diffeomorphism of $\mathbb H$ whose complex dilatation belongs to $M_B(\mathbb H)$.
Moreover, $H$ is bi-Lipschitz with respect to the hyperbolic metric. 
\end{theorem}

We note that in the case where the BMO norm of $\log h'$ is sufficiently small for $h \in \rm SQS$, 
Semmes \cite[Proposition 4.2]{Se} used a modified Beurling--Ahlfors extension $H$
by compactly supported smooth kernels $\phi$ and $\psi$ to prove the same properties as in Theorem \ref{FKP}.
By dividing the weight $h'$ into small pieces and composing the resulting maps, the assumption on the small BMO norm can be
removed to obtain a quasiconformal extension of the same properties.

A locally integrable non-negative measurable function $\omega \geq 0$ on $\mathbb R$ 
is called a weight. We say that $\omega$ is an {\it $A_\infty$-weight} if it satisfies 
the reverse Jensen inequality, namely,
there exists a constant $C_\infty \geq 1$ such that
\begin{equation}\label{iff}
\frac{1}{|I|} \int_I \omega(x) dx \leq C_\infty \exp \left(\frac{1}{|I|} \int_I \log \omega(x) dx \right) 
\end{equation}
for every bounded interval $I \subset \mathbb R$. 
On the contrary, $\omega$ is defined to be an $A_\infty$-weight if 
$h(x)=\int_0^x \omega(t)dt$ is a strongly quasisymmetric homeomorphism of $\mathbb R$ by
Coifman and Fefferman \cite[Theorem III]{CF}.
In other words, $h'$ is an $A_\infty$-weight if $h \in \rm SQS$.
It is known that these definitions are equivalent 
(see Hru\v{s}\v{c}ev \cite{Hr}).
Moreover, the constants $K$ and $\alpha$ for the strong quasisymmetry in \eqref{alphaK} can be estimated in terms of $C_\infty$.
Concerning the relationship with $A_p$-weight $(p>1)$ of Muckenhoupt \cite{M}, 
see Garc\'ia-Cuerve and Rubio de Francia \cite[Section IV.2]{GR} and Garnett \cite[Section VI.6]{Ga}. In particular, $\omega$ is an $A_\infty$-weight if and only if
it is an $A_p$-weight for all sufficiently large $p$.


We see that if $\omega$ is an $A_\infty$-weight on $\mathbb R$, then
$\log \omega$ belongs to ${\rm Re}\,\mathrm{BMO}(\mathbb R)$,
which is the real subspace of ${\rm BMO}(\mathbb R)$ consisting of all
real-valued BMO functions. In particular, 
$\log h' \in {\rm Re}\,\mathrm{BMO}(\mathbb R)$ for $h \in \rm SQS$.
Conversely, we know the following fact
(see \cite[p.409]{GR} and \cite[Lemma VI.6.5]{Ga}).

\begin{proposition}\label{C_0}
Suppose that a weight $\omega \geq 0$ satisfies $\log \omega \in {\rm Re}\,\mathrm{BMO}(\mathbb R)$.
If the BMO norm $\Vert \log \omega \Vert_{\mathrm{BMO}}$ 
is less than the constant $C_{JN}$ of \eqref{JN}, then $\omega$ is an $A_\infty$-weight. 
\end{proposition}

There is an example of $\phi \in {\rm Re}\,\mathrm{BMO}(\mathbb R)$ such that $e^\phi$ is not an $A_\infty$-weight:
$\phi(x)=\log( 1/|x|)$.
Let $\mathrm{BMO}^*(\mathbb R)$ denote the proper subset of $\mathrm{BMO}(\mathbb R)$
consisting of all BMO functions $\phi$ such that $e^{{\rm Re}\,\phi}=|e^\phi|$ is an $A_\infty$-weight.
Moreover, let
${\rm Re}\,\mathrm{BMO}^*(\mathbb R)={\rm Re}\,\mathrm{BMO}(\mathbb R) \cap \mathrm{BMO}^*(\mathbb R)$.
We have the following claim.

\begin{proposition}\label{convex2}
${\rm Re}\,\mathrm{BMO}^*(\mathbb R)$ is a convex open subset of the real Banach subspace ${\rm Re}\,\mathrm{BMO}(\mathbb R)$.
Hence, so is
$\mathrm{BMO}^*(\mathbb R)$ of the complex Banach space $\mathrm{BMO}(\mathbb R)$.
\end{proposition}

\begin{proof}
For the convexity of ${\rm Re}\,\mathrm{BMO}^*(\mathbb R)$, 
we have to show that
if $\omega$ and $\tilde \omega$ are $A_\infty$-weights, then $\omega^s\tilde \omega^t$ is also
an $A_\infty$-weight for $s,t \geq 0$ with $s+t=1$. 
By the H\"older inequality and \eqref{iff}, we have
\begin{align*}
\frac{1}{|I|} \int_I \omega(x)^s\tilde\omega(x)^tdx  
&\leq \left(\frac{1}{|I|} \int_I \omega(x)dx\right)^{s}\left(\frac{1}{|I|} \int_I \tilde\omega(x)dx\right)^{t}\\
&\leq \left(C_\infty \exp \left(\frac{1}{|I|} \int_I \log \omega(x) dx \right) \right)^s
\left(\widetilde C_\infty \exp \left(\frac{1}{|I|} \int_I \log \tilde \omega(x) dx \right) \right)^t\\
&= (C_\infty)^s (\widetilde C_\infty)^t \exp \left(\frac{1}{|I|} \int_I \log (\omega(x)^s \tilde \omega(x)^t) dx \right)
\end{align*}
for any bounded interval $I \subset \mathbb R$. Hence, $\omega^{s} \tilde\omega^{t}$ is an $A_\infty$-weight.

As another property of $A_\infty$-weights, we know that if $\omega$ is an $A_p$-weight, then
there is some $\varepsilon >0$ such that $\omega^r$ is an $A_p$-weight for every $r \in [0,1+\varepsilon)$
(see \cite[Theorem IV.2.7]{GR}). 
Combining these properties with the fact in Proposition \ref{C_0} that the open ball centered at the origin of 
${\rm Re}\,{\rm BMO}(\mathbb R)$
with radius $C_{JN}$ is contained in ${\rm Re}\,\mathrm{BMO}^*(\mathbb R)$, we can 
prove that ${\rm Re}\,\mathrm{BMO}^*(\mathbb R)$ is open. Indeed, ${\rm Re}\,\mathrm{BMO}^*(\mathbb R)$ is the union of
open cones spanned by the $C_{JN}$-neighborhood of the origin
having any points of ${\rm Re}\,\mathrm{BMO}^*(\mathbb R)$ as their vertices. 

Because $\mathrm{BMO}^*(\mathbb R)={\rm Re}\,{\rm BMO}^*(\mathbb R) \oplus i\,{\rm Re}\,{\rm BMO}(\mathbb R)$,
we also see that $\mathrm{BMO}^*(\mathbb R)$ is convex and open.
\end{proof}

A strongly quasisymmetric homeomorphism $h:\mathbb R \to \mathbb R$ can be also characterized by the 
composition operator on the Banach space ${\rm BMO}(\mathbb R)$. 
The pre-composition of $h$ to $\phi \in {\rm BMO}(\mathbb R)$ gives a change of the variables, 
and we denote
this linear operator on ${\rm BMO}(\mathbb R)$ by $C_h$.
Its boundedness is proved by Jones \cite{Jo} as follows. 
Concerning the dependence of the constants, see Gotoh \cite[Example 2.3]{Got}.

\begin{theorem}\label{pullback}
An increasing homeomorphism $h$ of $\mathbb R$ onto itself belongs to $\rm SQS$ if and only if 
the composition operator $C_h: \phi \mapsto \phi \circ h$ gives an automorphism of ${\rm BMO}(\mathbb R)$,
that is, $C_h$ and $C_h^{-1}$ are bounded linear operators. 
Moreover, the operator norm satisfies an estimate
$$
\Vert C_h \Vert \asymp \Vert C_h^{-1} \Vert \lesssim \frac{K}{\alpha}
$$
in terms of the constants $K$ and $\alpha$
for the strong quasisymmetry of $h$ in \eqref{alphaK}.
\end{theorem}

\section{Chord-arc curves and conformal welding}\label{section4}

The universal Teich\-m\"ul\-ler space $T$ is also identified with
the set of all normalized quasicircles. As the corresponding characterization for
the BMO Teich\-m\"ul\-ler space $T_B$, a certain geometric condition is obtained by Bishop and Jones \cite[Theorem 4]{BJ},
which is preserved under a
bi-Lipschitz self-homeomorphism of $\mathbb C$. This is a sort of localization of chord-arc condition,
and a chord-arc curve defined below satisfies this condition.
The subset consisting of all chord-arc curves occupies a certain portion of $T_B$.

\begin{definition}
A Jordan curve $\Gamma$ in $\mathbb C$ 
passing through $\infty$ is called a {\it chord-arc curve} if $\Gamma$ is locally rectifiable and there exists
a constant $\kappa \geq 1$ such that the length of the arc between any two points $z_1, z_2 \in \Gamma$
is bounded by $\kappa|z_1-z_2|$. In other words,
the arc length parametrization of $\Gamma$ yields a bi-Lipschitz embedding of $\mathbb R$ into $\mathbb C$.
\end{definition}

A Jordan curve $\Gamma$ in $\mathbb C$ 
passing through $\infty$ is a 
quasicircle if $\Gamma$ is the image of $\mathbb R$ under a quasiconformal self-homeomorphism of $\mathbb C$.
This is known to be equivalent to
satisfying a similar condition to the above definition, but with the length of the arc between $z_1$ and $z_2$
replaced by its diameter, so that 
$\Gamma$ is not even required to be locally rectifiable (see Ahlfors \cite[Theorems IV.4, 5]{Ah}). 
In particular, a chord-arc curve is a quasicircle. The corresponding characterization of a chord-arc curve by
the image of $\mathbb R$ was shown in Jerison and Kenig \cite[Proposition 1.13]{JK}, 
Pommerenke \cite[Theorems 7.9, 7.10]{Pom}, and Tukia \cite[Theorem]{Tu} as follows.

\begin{proposition}\label{biLip}
A Jordan curve $\Gamma$ passing through $\infty$ is a chord-arc curve if and only if $\Gamma$ is 
the image of $\mathbb R$ under a bi-Lipschitz self-homeomorphism $G$ of $\mathbb C$ with respect to the Euclidean metric.
In fact, any bi-Lipschitz embedding $\gamma:\mathbb R \to \mathbb C$ extends to a bi-Lipschitz self-homeomorphism of $\mathbb C$.
\end{proposition}

First, we note a basic property of the boundary extension  
of a conformal homeomorphism of $\mathbb H$ determined by $\mu \in M_B(\mathbb H)$ in general.
Let $F^\mu$ be the normalized quasiconformal self-homeo\-morphism of $\mathbb C$
that is conformal on $\mathbb H^-$ and has the complex dilatation $\mu$ on $\mathbb H^+$.
By Proposition \ref{L-holo}, the condition $\mu \in M_B(\mathbb H^+)$ implies $\log (F^\mu|_{\mathbb H^-})' \in 
\mathrm{BMOA}(\mathbb H^-)$.
Moreover, the converse is also true (see Theorem \ref{char}).

\begin{lemma}\label{app}
If $\Phi=\log (F^\mu|_{\mathbb H^-})'$ belongs to 
$\mathrm{BMOA}(\mathbb H^-)$,
then
$f=F^\mu|_{\mathbb R}$ has its derivative with $f'(x) \neq 0$ almost everywhere on $\mathbb R$, and
$\log f'$ coincides with $E(\Phi) \in \mathrm{BMO}(\mathbb R)$. 
\end{lemma}

\begin{proof}
By Theorem \ref{131},
the boundary extension $\phi=E(\Phi)$ is in $\mathrm{BMO}(\mathbb R)$, and in particular,
$\phi(x)$ is finite almost everywhere on $\mathbb R$. Since $f(\mathbb R)$ is a quasicircle,
it is known that $f'(x)$ exists and coincides with the angular
derivative of $F^\mu$ at $x$ almost everywhere on $\mathbb R$ (see \cite[Theorem 5.5]{Pom}). 
Hence, $\log f'=\phi$.
\end{proof}

In this setting of quasicircle,
the {\it Lavrentiev theorem} in particular gives a condition under which the image of $\mathbb R$ by $F^\mu$
is a chord-arc curve. See \cite[Theorem 4.2]{JK} and \cite[Theorem 7.11]{Pom}.
We note that every function in $\mathrm{BMOA}(\mathbb H^-)$ can be represented by the Poisson integral of
its boundary extension on $\mathbb R$ by Theorem \ref{decomposition}.

\begin{theorem}\label{Lavrentiev}
For a quasiconformal self-homeomorphism $F^\mu$ of $\mathbb C$ with $\mu \in M(\mathbb H^+)$,
$\Gamma=F^\mu(\mathbb R)$ is a chord-arc curve if and only if $\log (F^\mu|_{\mathbb H^-})'$ belongs to $\mathrm{BMOA}(\mathbb H^-)$,
$f=F^\mu|_{\mathbb R}$ is locally absolutely continuous, and $|f'|$ is an $A_{\infty}$-weight.
Namely, the equivalent condition is that $\log f' \in \mathrm{BMO}^*(\mathbb R)$.
\end{theorem}

In fact, any chord-arc curve falls in this situation as 
the following claim asserts. 
Arguments and expositions around these results in a more general setting can be found in MacManus \cite{Mac}.

\begin{corollary}\label{BJcondition}
Every chord-arc curve passing through $\infty$ is the image of $\mathbb R$ under some $F^\mu$
with $\log f' \in \mathrm{BMO}^*(\mathbb R)$ for $f=F^\mu|_{\mathbb R}$.
\end{corollary} 

\begin{proof}
By Proposition \ref{biLip}, $\Gamma$ is the image of $\mathbb R$ under some quasiconformal self-homeo\-morphism $G$ of $\mathbb C$.
We choose $F^\mu=G \circ H$ so that it is conformal on $\mathbb H$ by pre-composing a quasiconformal self-homeomorphism $H$ of $\mathbb C$ preserving $\mathbb R$. Then, Theorem \ref{Lavrentiev} yields the assertion.
\end{proof}

\begin{definition}
A subset of $T_B$ consisting of all elements $[\mu]$ such that
$F^\mu(\mathbb R)$ is a chord-arc curve is denoted by $T_C$.
\end{definition}

There exists some $[\mu] \in T_B$ that is not contained in $T_C$.
In fact, there are examples of $F^\mu$ such that $F^\mu(\mathbb R)$
are not locally rectifiable (see Astala and Zinsmeister \cite[Theorem 6]{AZ}, Bishop \cite[Theorem 1.1]{Bi} and 
Semmes \cite[Theorem]{Se0}). 

\begin{proposition}\label{open}
$T_C$ is a proper open subset of $T_B$ containing the origin.
\end{proposition}

\begin{proof}
We consider
the composition of the pre-Bers embedding $\beta:T_B \to \mathrm{BMOA}(\mathbb H^-)$ and
the trace operator $E^-:\mathrm{BMOA}(\mathbb H^-) \to \mathrm{BMO}(\mathbb R)$.
Since $E^- \circ \beta:T_B \to \mathrm{BMO}(\mathbb R)$ is continuous, $\mathrm{BMO}^*(\mathbb R)$
is open in $\mathrm{BMO}(\mathbb R)$ by Proposition \ref{convex2}, and $T_C=(E^- \circ \beta)^{-1}(\mathrm{BMO}^*(\mathbb R))$
by Theorem \ref{Lavrentiev}, we see that $T_C$ is an open subset of $T_B$ containing the origin.
\end{proof}

\begin{remark}
Whether $T_C$ is connected or not is an open problem. Since the inverse image $E^{-1}(\mathrm{BMO}^*(\mathbb R))$ is
also a convex open subset of $\mathrm{BMOA}(\mathbb H)$, the shape of the image of $T_B$
under the pre-Bers embedding $\beta$ comes into question. 
\end{remark}

The following properties for $T_C$ are easily obtained. See MW \cite[Theorem 4]{WM-1}.

\begin{proposition}\label{Fenn}
$(1)$ Every element $[\mu] \in T_B$ can be obtained by a finite 
composition $[\mu]=[\mu_1] \ast \cdots \ast [\mu_n]$ of elements $[\mu_i] \in T_C$ $(i=1,\ldots ,n)$.
$(2)$ If $[\mu] \in T_C$ then $[\mu]^{-1} \in T_C$.
\end{proposition}

In general, a quasisymmetric homeomorphism $h:\mathbb R \to \mathbb R$ can be expressed as the discrepancy between the boundary values 
$f_1=F_1|_{\mathbb R}$ and $f_2=F_2|_{\mathbb R}$ of two conformal homeomorphisms $F_1:\mathbb H^- \to \Omega^-$ and 
$F_2:\mathbb H^+ \to \Omega^+$, where $\Omega^-$ and $\Omega^+$ are the complementary domains in $\mathbb C$ with
$\partial \Omega^-=\partial \Omega^+$. This expression $h=f_2^{-1} \circ f_1$ is called {\it conformal welding}.
This allows us to see that properties of $h$
are determined by those of $f_1$ and $f_2$.

We determine the class of $h \in {\rm QS}$
whose quasiconformal extension $H(\mu)$ to $\mathbb H$ is given by 
$\mu \in M_B(\mathbb H)$. The following theorem corresponds to the ``only if'' part of Theorem \ref{SQS}.
Representing a quasisymmetric homeomorphism $h(\mu)=H(\mu)|_{\mathbb R}$ with $\mu \in M_B(\mathbb H)$
by conformal welding, we will give an alternative proof for it.

\begin{theorem}\label{main1}
If $\mu \in M_B(\mathbb H)$,
then $h(\mu) \in \rm SQS$. In other words, 
$h=h(\mu)$ is locally absolutely continuous and 
$\phi=\log h'$ belongs to $\mathrm{BMO}^*(\mathbb R)$.
\end{theorem}

The following method of conformal welding works only when $[\mu]$ belongs to $T_C$, that is to say, when the conformal welding is done along
a chord-arc curve $F^\mu(\mathbb R)$. For the general case, we decompose
$[\mu]$ into finitely many such elements, and apply this to each of them.

\begin{lemma}\label{smalldil}
If 
$F^\mu(\mathbb R)$ is a chord-arc curve,
then $h(\mu)$ is strongly quasisymmetric.
\end{lemma}

\begin{proof}
Let $F_1=F^\mu$ as before and let $F_2=F_{\bar \mu^{-1}}$ be the normalized quasiconformal self-homeomorphism of $\mathbb C$
that is conformal on $\mathbb H^+$ and has complex dilatation $\bar \mu^{-1}$ on $\mathbb H^-$. Here, $\bar \mu$ is
the reflection of $\mu$ with respect to $\mathbb R$ defined by $\bar \mu( \bar z)=\mu(z)$.
We note that the normalization of $F_1$ and $F_2$ requires $F_1(\mathbb R)=F_2(\mathbb R)$.
Let $f_1=F_1|_{\mathbb R}$ and $f_2=F_2|_{\mathbb R}$. 
Then, we have that $f_2 \circ h=f_1$. 

Theorem \ref{Lavrentiev} asserts that $f_1$ and $f_2$ are locally absolutely continuous 
and $\log f_1'$ and $\log f_2'$ belong to $\mathrm{BMO}^*(\mathbb R)$. Namely, $|f_1'|$ and $|f_2'|$ are 
$A_\infty$-weights. Moreover,
$f_2'(x) \neq 0\ {\rm (a.e.)}$ by Lemma \ref{app}, and this implies that $h$ is also locally absolutely continuous.
Then, taking the absolute value of the derivative for the conformal welding $f_2 \circ h=f_1$, we have
$|f_2'| \circ h \cdot h'=|f_1'|$.

Let $\tilde f_1(x)=\int_0^x |f_1'(t)|dt$ and $\tilde f_2(x)=\int_0^x |f_2'(t)|dt$, which belong to $\rm SQS$.
Moreover,
$$
\tilde f_1(x)=\int_0^x |f_2'| \circ h(t) \cdot h'(t)dt=\int_0^{h(x)} |f_2'(\tau)|d\tau=\tilde f_2 \circ h(x).
$$
Hence, $h=\tilde f_2^{-1} \circ \tilde f_1$, which is also a strongly quasisymmetric homeomorphism.
\end{proof}

\begin{proof}[Proof of Theorem \ref{main1}]
We represent $h=h(\mu)$ for $\mu \in M_B(\mathbb H^+)$
by conformal welding as $h=(F_{\bar \mu^{-1}})^{-1} \circ F^\mu|_{\mathbb R}$.
If $F^\mu(\mathbb R)$ is a chord-arc curve, then Lemma \ref{smalldil} implies that $h \in {\rm SQS}$.
In the general case, we decompose $[\mu] \in T_B$ into a finite number of elements in $T_C$ by Proposition \ref{Fenn};
$[\mu]=[\mu_1] \ast \cdots \ast [\mu_n]$ for $[\mu_i] \in T_C$ $(i=1,\ldots, n)$. 
By the above argument, each $h_i=h(\mu_i)$ is in ${\rm SQS}$.
Hence, we see that $h=h_1 \circ \cdots \circ h_n$ is strongly quasisymmetric.
\end{proof}

Finally, we mention the characterizations for a Beltrami coefficient $\mu$
to be in $M_B(\mathbb H)$ in terms of the Schwarzian and the pre-Schwarzian derivative maps
$S$ and $L$. 

\begin{theorem}\label{char}
Suppose that $F^\mu(\mathbb R)$ is a chord-arc curve for $\mu \in M(\mathbb H^+)$. 
Then, the following conditions are equivalent:
$(a)$ $\mu \in M_B(\mathbb H^+)$; $(b)$ $S(\mu) \in A_B(\mathbb H^-)$;
$(c)$ $L(\mu) \in {\rm BMOA}(\mathbb H^-)$. 
\end{theorem}

\begin{proof}
The implications $(a) \Rightarrow (b) \Rightarrow (c)$ are along the same line as
Propositions \ref{S-holo} and \ref{L-holo}. The crucial step is to prove $(c) \Rightarrow (a)$.
If $F^\mu(\mathbb R)$ is a chord-arc curve, Lemma \ref{smalldil} shows that
$h(\mu)$ is strongly quasisymmetric.
Then, by ``if'' part of Theorem \ref{SQS} (=Theorem \ref{converse}),
we conclude that $[\mu] \in T_B$, that is, $\mu$ belongs to $M_B(\mathbb H^+)$.
\end{proof}

\begin{remark}
The above theorem is still valid without the assumption that $F^\mu(\mathbb R)$ is a chord-arc curve.
This is proved in Astala and Zinsmeister \cite[Theorem 4]{AZ}. The step $(c) \Rightarrow (a)$
relies on the geometric characterization of a curve $F^\mu(\mathbb R)$
for $\log (F^\mu|_{\mathbb H^-})' \in {\rm BMOA}(\mathbb H^-)$ given by Bishop and Jones \cite[Theorem 4]{BJ}.
As a consequence, 
it follows that
$$
S(M_B(\mathbb H^+))=S(M(\mathbb H^+)) \cap A_B(\mathbb H^-),\quad
L(M_B(\mathbb H^+))=L(M(\mathbb H^+)) \cap {\rm BMOA}(\mathbb H^-).
$$
\end{remark}

\section{BMO Embeddings in Bers Coordinates}\label{section5}

We generalize strongly quasisymmetric homeomorphisms $h:\mathbb R \to \mathbb R$ to 
BMO embeddings $\gamma:\mathbb R \to \mathbb C$, and consider those whose images are chord-arc curves.
Then, we use the BMO Teich\-m\"ul\-ler space $T_B$ to coordinate these embeddings. 
We remark here that a BMO embedding is a mapping $\gamma$ of $\mathbb R$, that is, 
the image $\Gamma=\gamma(\mathbb R)$ together with its parametrization,
whereas a chord-arc curve refers to the image $\Gamma$ itself
of a certain special BMO embedding $\gamma$. 

\begin{definition}
A topological embedding $\gamma:\mathbb R \to \mathbb C$ passing through $\infty$ is called a
{\it BMO embedding} if there is a quasiconformal self-homeomorphism $G$ of $\mathbb C$ with $G|_{\mathbb R}=\gamma$
whose complex dilatation $\mu=\bar \partial G/\partial G$ satisfies
$\mu|_{\mathbb H^+} \in M_B(\mathbb H^+)$ and $\mu|_{\mathbb H^-} \in M_B(\mathbb H^-)$.
\end{definition}

We first consider the derivative of a BMO embedding $\gamma:\mathbb R \to \mathbb C$ though $\gamma$ is not necessarily
absolutely continuous.
The following claim is proved in a more general setting in MacManus \cite[Theorem 6.2]{Mac}.
See WM \cite[Proposition 3.3]{WM-0} for the proof along quasiconformal theory.
Later in Theorem \ref{main-b}, 
we show this in the special case where the image of $\gamma$ is a chord-arc curve.

\begin{proposition}\label{curve}
A BMO embedding 
$\gamma:\mathbb R \to \mathbb C$ has its derivative $\gamma'$ almost everywhere on $\mathbb R$
and $\log \gamma'$
belongs to ${\rm BMO}(\mathbb R)$. 
\end{proposition}

We consider parametrization of all BMO embeddings.
The simultaneous uniformization due to Bers works for it. In general,
a quasisymmetric embedding $\gamma:\mathbb R \to \Gamma \subset \mathbb C$ 
onto a quasicircle $\Gamma$ passing through $\infty$ is induced by a quasiconformal 
self-homeomorphism $G(\mu^+,\mu^-)$ of $\mathbb C$
for $\mu^+ \in M(\mathbb H^+)$ and $\mu^- \in M(\mathbb H^-)$. We assume that $G(\mu^+,\mu^-)$ is normalized
so that it fixes $0$, $1$, and $\infty$. 
We see in the following proposition that
such an embedding $\gamma=G(\mu^+,\mu^-)|_{\mathbb R}$ is determined 
by the Teich\-m\"ul\-ler equivalence class pair $([\mu^+],[\mu^-])$.
The proof is the same as that in \cite[Proposition 4.1]{WM-0}.

\begin{proposition}\label{well-def}
For $\mu^+,\,\nu^+ \in M(\mathbb H^+)$ and $\mu^-,\,\nu^- \in M(\mathbb H^-)$,
$G(\mu^+,\mu^-)|_{\mathbb R}=G(\nu^+,\nu^-)|_{\mathbb R}$ if and only if $[\mu^+]=[\nu^+]$ and
$[\mu^-]=[\nu^-]$ in $T$.
\end{proposition}

Thus, the space of all such normalized quasisymmetric embeddings can be 
identified with the product $T^+ \times T^-$ of the universal Teich\-m\"ul\-ler space
for $T^+=\pi(M(\mathbb H^+))$ and $T^-=\pi(M(\mathbb H^-))$.
We refer to this as {\it the Bers coordinates}. 
We can introduce these coordinates for the BMO embeddings as well. 
Let $T_B^+=\pi(M_B(\mathbb H^+))$ and $T_B^-=\pi(M_B(\mathbb H^-))$.
Any BMO embedding is represented by $\gamma([\mu^+],[\mu^-])=G(\mu^+,\mu^-)|_{\mathbb R}$
for $([\mu^+],[\mu^-]) \in T_B^+ \times T_B^-$, and thus
$T_B^+ \times T_B^-$ becomes the Bers coordinates of the space of normalized BMO embeddings.

Let $\bar \mu \in M_B(\mathbb H^-)$ denote the reflection $\overline{\mu(\bar z)}$ 
of a Beltrami coefficient $\mu(z)$ for $z \in \mathbb H^+$.
Then, $G(\mu, \bar \mu)$ is nothing but the normalized quasiconformal homeomorphism $H(\mu):\mathbb C \to \mathbb C$ preserving $\mathbb R$.
The {\it axis of symmetry} of the product space $T_B^+ \times T_B^-$ is defined as
$$
{\rm Sym}\,(T_B^+ \times T_B^-)=\{([\mu],[\bar \mu]) \mid [\mu] \in T_B^+\}.
$$
For $[\mu] \in T_B$, the corresponding $h=h(\mu) \in \rm SQS$ is expressed as 
$h=G(\mu, \bar \mu)|_{\mathbb R}=\gamma([\mu],[\bar \mu])$. 
The canonical map $\iota:T_B \to {\rm Sym}\,(T_B^+ \times T_B^-) \subset T_B^+ \times T_B^-$
defined by $[\mu] \mapsto ([\mu],[\bar \mu])$ is a real-analytic embedding, and hence
${\rm Sym}\,(T_B^+ \times T_B^-)$ is a real-analytic submanifold of $T_B^+ \times T_B^-$.

For every $[\nu] \in T_B$, the right translation $r_{[\nu]}:T_B \to T_B$ of the group structure of $T_B \cong \rm SQS$
is defined by $[\mu] \mapsto [\mu] \ast [\nu]$ for every $[\mu] \in T_B$.
It is known that $r_{[\nu]}$ is a biholomorphic automorphism of $T_B$ (see Shen and Wei \cite[Remark 5.1]{SWei} and \cite[Lemma 4.2]{WM-0}).
This can be extended to $T_B^+ \times T_B^-$ as the parallel translation
$$
R_{[\nu]}([\mu^+],[\mu^-]) = (r_{[\nu]}([\mu^+]), r_{[\bar \nu]}([\mu^-]))=([\mu^+] \ast [\nu], [\mu^-] \ast [\bar \nu]),
$$
which is a biholomorphic automorphism of $T_B^+ \times T_B^-$ that preserves  
${\rm Sym}\,(T_B^+ \times T_B^-)$.

We see that
any chord-arc curve $\Gamma$ passing through $\infty$ is the image of some BMO embedding by Corollary \ref{BJcondition}
and Theorem \ref{char}.
We have defined $T_C$ as the subset of $T_B$ consisting of all $[\mu] \in T_B$ such that $F^\mu(\mathbb R)$ is
a chord-arc curve, which is an open subset of $T_B$ by Proposition \ref{open}.
Let $\widetilde T_C$ be the subset of $T_B^+ \times T_B^-$ consisting of the Bers coordinates of BMO
embeddings $\gamma:\mathbb R \to \mathbb C$ whose images
are chord-arc curves. By definition, $T_C^+ \times \{[0]\}$, $\{[0]\} \times T_C^-$, and ${\rm Sym}\,(T_B^+ \times T_B^-)$ are
all contained in $\widetilde T_C$, where $T_C^+=T_C$ is the subset of $T_B^+=T_B$, and $T_C^-=\{[\bar \mu] \in T_B^- \mid
[\mu] \in T_C\}$.
We will see later that $\widetilde T_C$ is an open subset of $T_B^+ \times T_B^-$.

In Theorem \ref{main1}, by using the method of conformal welding, we have obtained that 
any element $h=\gamma([\mu], [\bar \mu])$ with $[\mu] \in T_B$ is
locally absolutely continuous and $\log h'$ belongs to ${\rm Re}\,\mathrm{BMO}^*(\mathbb R)$.
The characterization of a chord-arc curve as in
Theorem \ref{Lavrentiev} can be generalized as follows by using the inverse method of conformal welding. 

\begin{theorem}\label{main-b}
Let $\gamma=\gamma([\mu^+],[\mu^-])$ be a BMO embedding with  
$([\mu^+],[\mu^-]) \in T_B^+ \times T_B^-$. Then, the image of $\gamma:\mathbb R \to \mathbb C$ is a chord-arc curve,
that is $([\mu^+],[\mu^-]) \in \widetilde T_C$,
if and only if $\gamma$ is locally absolutely continuous and
$\log \gamma'$ belongs to $\mathrm{BMO}^*(\mathbb R)$.
\end{theorem}

\begin{proof}
We represent $\gamma$ as the following composition:
$$
\gamma([\mu^+],[\mu^-])=\gamma([0],[\mu^-]\ast[\overline{\mu^+}]^{-1}) \circ \gamma([\mu^+],[\overline{\mu^+}]).
$$
Here, $f=\gamma([0],[\mu^-]\ast[\overline{\mu^+}]^{-1})$ is given as $F_{\nu}|_{\mathbb R}$
for the normalized quasiconformal self-homeo\-morphism $F_{\nu}$ of $\mathbb C$ that is conformal 
on $\mathbb H^+$ and has complex dilatation $\nu$ on $\mathbb H^-$ with $[\nu]=[\mu^-]\ast[\overline{\mu^+}]^{-1}$, and
$h=\gamma([\mu^+],[\overline{\mu^+}])$ is given as $H(\mu^+)|_{\mathbb R}$ for the normalized quasiconformal self-homeomorphism $H(\mu^+)$
of $\mathbb C$ with the indicated complex dilatation. 
We have $[\nu] \in T_B$ since $T_B \cong {\rm SQS}$ is a group. Hence, $f$
is a BMO embedding and
the image of $f$ coincides with that of $\gamma$.

Suppose that the image of $\gamma$ is a chord-arc curve.
Then, $f$ is locally absolutely continuous and $\log f' \in \mathrm{BMO}^*(\mathbb R)$ by
Theorem \ref{Lavrentiev}. In addition, $h \in \rm SQS$
by Theorem \ref{main1}.
Therefore, $\gamma=f \circ h$ is locally absolutely continuous and satisfies
$|\gamma'|=|f' |\circ h \cdot h'$.
Let $\tilde f(x)=\int^x_0 |f'(t)|dt$, which is in $\rm SQS$. Then,
$$
\tilde \gamma(x)=\int^x_0 |\gamma'(t)|dt=\int^x_0 |f' |\circ h(t) \cdot h'(t)dt
=\int^{h(x)}_0 |f'(\tau) |d\tau=\tilde f \circ h(x)
$$  
is also in $\rm SQS$. This implies that $|\gamma'|$ is an $A_\infty$-weight, and
hence $\log \gamma' \in \mathrm{BMO}^*(\mathbb R)$.

Conversely, suppose that $\gamma$ is locally absolutely continuous and
$\log \gamma'$ is in $\mathrm{BMO}^*(\mathbb R)$. Then, $f$ is also locally absolutely continuous.
Under the above definition, we consider $\tilde \gamma$ and $\tilde f$. Since $|\gamma'|$ is an $A_\infty$-weight,
$\tilde \gamma$ is in $\rm SQS$, and hence so is $\tilde f$. Thus, $|f'|$ is an $A_\infty$-weight and
$\log f'$ is in $\mathrm{BMO}^*(\mathbb R)$. By Theorem \ref{Lavrentiev} again, the image of $f$ is a
chord-arc curve. This implies that the image of $\gamma$ is also a chord-arc curve.
\end{proof}

In particular, if a BMO embedding $\gamma$ satisfies $|\gamma'|=1$, which means that
$\gamma$ is parametrized by its arc-length, then $\gamma(\mathbb R)$ is a chord-arc curve.

\section{Holomorphy to the BMO space}\label{section6}

For a BMO embedding $\gamma:\mathbb R \to \mathbb C$ with
$\gamma=\gamma([\mu^+],[\mu^-])$ for $([\mu^+],[\mu^-]) \in T_B^+ \times T_B^-$,
we have $\log \gamma' \in \mathrm{BMO}(\mathbb R)$ by Proposition \ref{curve}.
By this correspondence, we define a map
$$
\Lambda:T_B^+ \times T_B^- \to \mathrm{BMO}(\mathbb R).
$$
We first prove that this correspondence is holomorphic in the Bers coordinates.
This is shown in WM \cite[Theorem 4.3]{WM-0}.

\begin{theorem}\label{holomorphic}
The map $\Lambda:T_B^+ \times T_B^- \to \mathrm{BMO}(\mathbb R)$ 
is holomorphic.
\end{theorem}

\begin{proof}
By the Hartogs theorem for Banach spaces (see Chae \cite[\S 14.27]{Ch}), to see that $\Lambda$ is holomorphic,
it suffices to show that $\Lambda$ is separately holomorphic.
Namely, we fix, say $[\mu^+_0]$, and prove that $\Lambda([\mu^+_0],[\mu^-])$ is holomorphic on 
$[\mu^-]$. The other case is treated in the same way.

Let $C_{h_0}$ be
the composition operator acting on $\mathrm{BMO}(\mathbb R)$ induced by $h_0=h(\mu^+_0) \in \rm SQS$. We define
the affine translation $Q_{h_0}(\phi)$ of $\phi \in \mathrm{BMO}(\mathbb R)$ by $C_{h_0}(\phi)+\log h_0'$.
Then, $\Lambda \circ R_{[\mu^+_0]} = Q_{h_0} \circ \Lambda$ holds. 
This relation yields
a useful representation
\begin{equation}\label{translation}
\Lambda([\mu^+_0],\,\cdot \,)=Q_{h_0} \circ \Lambda([0],r_{[\,\bar{\mu}^+_0\,]}^{-1}(\,\cdot \,)) .
\end{equation}
Here, 
$\Lambda([0],\,\cdot \,)$ becomes the trace operator $E^+$ in Theorem \ref{131} by
composing the pre-Bers embedding $\beta^-:T_B^- \to \mathrm{BMOA}(\mathbb H^+)$. Namely,
$$
\Lambda([0],[\mu])=E^+(\beta^-([\mu])) \quad ([\mu] \in T_B^-).
$$
Since $E^+$ is a bounded linear operator and $r_{[\,\bar{\mu}^+_0\,]}^{-1}$ is holomorphic, we conclude that
$\Lambda([\mu^+_0],\,\cdot \,)$ is holomorphic.
\end{proof}

In the above proof, we see that the affine translation
$Q_{h_0}$ of $\mathrm{BMO}(\mathbb R)$ preserves the image $\Lambda(\widetilde T_C)$
because $R_{[\mu^+_0]}$ preserves $\widetilde T_C$.

\begin{proposition}
$\widetilde T_C=\Lambda^{-1}(\mathrm{BMO}^*(\mathbb R))$, and $\widetilde T_C$ is an open subset of $T_B^+ \times T_B^-$.
\end{proposition}

\begin{proof}
The first assertion follows from Theorem \ref{main-b}. Since $\Lambda$ is continuous by Theorem \ref{holomorphic}
and $\mathrm{BMO}^*(\mathbb R)$ is open by Proposition \ref{convex2},
we obtain the second assertion.
\end{proof}

In the sequel, we restrict $\Lambda$ to $\widetilde T_C$ and define this map as
$$
\Lambda:\widetilde T_C \to \mathrm{BMO}^*(\mathbb R)
$$
using the same notation $\Lambda$.

\begin{proposition}\label{inj}
The map $\Lambda:\widetilde T_C \to \mathrm{BMO}^*(\mathbb R)$ 
is a holomorphic injection.
\end{proposition}

\begin{proof}
Holomorphy follows from Theorem \ref{holomorphic}.
Let $\Lambda([\mu_1^+],[\mu_1^-])=\log \gamma_1'$, $\Lambda([\mu_2^+],[\mu_2^-])=\log \gamma_2'$, and
suppose that $\log \gamma_1'=\log \gamma_2'$. Since
$\gamma_1=\gamma([\mu_1^+],[\mu_1^-])$ and $\gamma_2=\gamma([\mu_2^+],[\mu_2^-])$ are locally absolutely continuous,
we have $\gamma_1=\gamma_2$ by the normalization.
This implies that $[\mu_1^+]=[\mu_2^+]$ and $[\mu_1^-]=[\mu_2^-]$ by Proposition \ref{well-def}, 
and hence $\Lambda$ is injective.
\end{proof}

We denote the tangent space of $T_B$ at $[\mu]$ by ${\mathscr T}_{[\mu]}T_B$.
The tangent space of $T_B^+ \times T_B^-$ at $([\mu^+],[\mu^-])$ is represented by the direct sum
\begin{equation}\label{tangentspace}
{\mathscr T}_{([\mu^+],[\mu^-])}(T_B^+ \times T_B^-)={\mathscr T}_{[\mu^+]}T_B^+ \oplus {\mathscr T}_{[\mu^-]}T_B^-.
\end{equation}
By the identification $T_B^+ \cong \beta(T_B^+) \subset \mathrm{BMOA}(\mathbb H^-)$
and $T_B^- \cong \beta(T_B^-) \subset \mathrm{BMOA}(\mathbb H^+)$ under the pre-Bers embedding by Theorem \ref{Bers},
we may assume that ${\mathscr T}_{[\mu^+]}T_B^+ \cong \mathrm{BMOA}(\mathbb H^-)$ and
${\mathscr T}_{[\mu^-]}T_B^- \cong \mathrm{BMOA}(\mathbb H^+)$.
Then, 
the derivative $d_{([\mu^+],[\mu^-])} \Lambda$ of $\Lambda$ at $([\mu^+],[\mu^-])$ is regarded as
the linear mapping
$$
d_{([\mu^+],[\mu^-])}\, \Lambda:\mathrm{BMOA}(\mathbb H^-) \oplus \mathrm{BMOA}(\mathbb H^+)
\to \mathrm{BMO}(\mathbb R)
=\mathrm{BMOA}(\mathbb H^-) \oplus \mathrm{BMOA}(\mathbb H^+)
$$
by taking the direct sum decomposition \eqref{omitE} into account.

The derivative $d_{([0],[0])}\,\Lambda$ at the origin can be easily understood.
By checking that the restriction of $\Lambda$ to $T_B^+$ and $T_B^-$ coincides with   
$$
\Lambda|_{T_B^+ \times \{[0]\}}=\beta^+:T_B^+ \to \mathrm{BMOA}(\mathbb H^-),\quad
\Lambda|_{\{[0]\} \times T_B^-}=\beta^-:T_B^- \to \mathrm{BMOA}(\mathbb H^+),
$$
we see that the derivative $d_{([0],[0])}\, \Lambda$ is the identity map of $\mathrm{BMOA}(\mathbb H^-) \oplus \mathrm{BMOA}(\mathbb H^+)$.
This implies the following claim by the inverse mapping theorem (see \cite[\S 7.18]{Ch}). 

\begin{proposition}\label{derivative0}
The inverse map $\Lambda^{-1}$ is holomorphic in
some neighborhood $U$ of $0$ in 
$\mathrm{BMO}^*(\mathbb R)$ with $U \subset \Lambda(\widetilde T_C)$.
\end{proposition}

Theorem \ref{main1} implies that $\Lambda([\mu], [\bar \mu]) \in {\rm Re}\, \mathrm{BMO}^*(\mathbb R)$ for every
$([\mu], [\bar \mu]) \in {\rm Sym}\,(T_B^+ \times T_B^-)$. The converse of this claim also holds.

\begin{theorem}\label{converse}
For every $\phi \in {\rm Re}\, \mathrm{BMO}^*(\mathbb R)$, there exists $([\mu], [\bar \mu]) \in {\rm Sym}\,(T_B^+ \times T_B^-)$
such that $\Lambda([\mu], [\bar \mu])=\phi$. Hence, any strongly quasisymmetric homeomorphism
$h(x)=\int_0^x e^{\phi(t)}dt$ is given as $h=h(\mu)$ for some $\mu \in M_B(\mathbb H^+)$. 
\end{theorem}

\begin{proof}
For each $s \in [0, 1]$, let 
\begin{equation*}
h_s(x) = \int_0^x e^{s\phi(t)} dt. 
\end{equation*}
Then, $h_s$ is an increasing and locally absolutely continuous homeomorphism of $\mathbb R$
with $h_0 = \rm id$ and $\log h_s' = s\phi \in {\rm Re}\, \mathrm{BMO}^*(\mathbb R)$. 
By the H\"older inequality, we see that the constant $C_\infty$ in \eqref{iff} for the $A_\infty$-weights
$h_s'$ for all $s$ are bounded by that for $h'$.
Hence, $h_s$ are uniformly strongly quasisymmetric homeomorphisms for all $s \in [0, 1]$
in the sense that we can take the constants $\alpha$ and $K$
in \eqref{alphaK} uniformly, and so are their inverses $h_s^{-1}$ by 
Coifman and Fefferman \cite[Lemma 5]{CF}. 

We have that
\begin{equation*}
\log (h_{\tilde s} \circ h_{s}^{-1})' = (\log h_{\tilde s}' - \log h_{s}')\circ h_s^{-1} = (\tilde s - s)C_{h_s}^{-1}\phi
\end{equation*}
for any $s, \tilde s \in [0, 1]$.
This belongs to ${\rm Re}\, \mathrm{BMO}(\mathbb R)$ and
the operator norms $\Vert C_{h_s}^{-1}\Vert$ are uniformly bounded by Theorem \ref{pullback} because the constants for strong quasisymmetry
of $h_s^{-1}$ are independent of $s$.
Hence, we can choose a positive integer $n \geq 1$ such that
$\log (h_{j/n}\circ h_{(j-1)/n}^{-1})'$ belong to the neighborhood $U$ of $0$ in $\mathrm{BMO}^*(\mathbb R)$
taken by Proposition \ref{derivative0} for all $j=1, \ldots, n$. 

By this proposition,
there exist $([\mu_j], [\bar \mu_j]) \in {\rm Sym}\,(T_B^+ \times T_B^-)$ such that
$\Lambda([\mu_j], [\bar \mu_j])=\log (h_{j/n}\circ h_{(j-1)/n}^{-1})'$ for all $j$.
Therefore,
\begin{align}
h=(h_1 \circ h_{(n-1)/n}^{-1}) \circ \cdots \circ  (h_{1/n} \circ h_0^{-1})
\end{align}
is given as $[\mu]=[\mu_n] \ast \cdots  \ast [\mu_1]$ in $T_B$.
Then, we have $\Lambda([\mu], [\bar \mu])=\log h'=\phi$, which is the desired assertion.
\end{proof}

Therefore, combined with Theorem \ref{converse},
Theorem \ref{main1} is improved to the complete characterization of $T_B$ 
in terms of strong quasisymmetry or ${\rm Re}\,\mathrm{BMO}^*(\mathbb R)$.
Thus, we obtain an alternate proof of Theorem \ref{SQS}, independently of the theorems in 
Fefferman, Kenig and Pipher \cite{FKP}.

\section{Biholomorphic correspondence}\label{section7}

We have seen that BMO embeddings $\gamma$ with chord-arc images are represented in the Bers coordinates, and the map $\Lambda$ to $\mathrm{BMO}^*(\mathbb{R})$ via $\log \gamma'$ is a holomorphic injection. In fact, this map is a biholomorphic homeomorphism onto its image, as shown in WM \cite[Theorem 6.1]{WM-0}. Here, we restructure the proof to emphasize the surjectivity of the derivative of $\Lambda$ as the main point.

\begin{theorem}\label{main2}
Let $\Lambda:\widetilde T_C \to \mathrm{BMO}^*(\mathbb R)$ be the holomorphic injection
given by
$$
\Lambda([\mu^+],[\mu^-])= \log \gamma' \quad (\gamma=\gamma([\mu^+],[\mu^-])).
$$
Then,
its image $\Lambda(\widetilde T_C)$ is an open subset of $\mathrm{BMO}^*(\mathbb R)$ containing
${\rm Re}\,\mathrm{BMO}^*(\mathbb R)$,
and $\Lambda$ is a biholomorphic homeomorphism onto $\Lambda(\widetilde T_C)$.
\end{theorem}

In the direct sum decomposition of the tangent space 
\begin{equation}\label{T-tangent}
{\mathscr T}_{([\mu^+],[\mu^-])}{\widetilde T}_C
= {\mathscr T}_{[\mu^+]}T_B^+ \oplus {\mathscr T}_{[\mu^-]}T_B^-,
\end{equation}
we denote the canonical projections by
$$
J^+:{\mathscr T}_{([\mu^+],[\mu^-])}(\widetilde T_C) \to {\mathscr T}_{[\mu^-]}(T_B^-),\quad
J^-:{\mathscr T}_{([\mu^+],[\mu^-])}(\widetilde T_C) \to {\mathscr T}_{[\mu^+]}(T_B^+).
$$
We first determine the image of each factor under the derivative $d_{([\mu^+],[\mu^-])}$.
In the sequel, $h^\pm \in \rm SQS$ always denote the quasisymmetric homeomorphisms given by $h^\pm=h(\mu^\pm)$.

\begin{lemma}\label{factor}
$d_{([\mu^+],[\mu^-])}\, \Lambda({\mathscr T}_{[\mu^\pm]}T_B^\pm)=C_{h^\mp}(\mathrm{BMOA}(\mathbb H^\mp))$.
\end{lemma}

\begin{proof}
Formula \eqref{translation} shows that
$$
\Lambda([\mu^+],[\mu^-])
=\Lambda \circ R_{[\mu^+]}([0],[\mu^-] \ast [\bar \mu^+]^{-1})
=Q_{h^+} \circ \Lambda([0],r_{[\bar \mu^+]}^{-1}([\mu^-])).
$$
By fixing $[\mu^+]$, we take the partial derivative of this formula along the direction of $T_B^-$. Then,
\begin{equation}
d_{([\mu^+],[\mu^-])}\, \Lambda|_{{\mathscr T}_{[\mu^-]}T_B^-}
=C_{h^+} \circ d_{([0],[\mu^-]\ast[\bar \mu^+]^{-1})}\Lambda \circ d_{[\mu^-]}r^{-1}_{[\bar \mu^+]},
\end{equation}
where $d_{([0],[\mu^-]\ast[\bar \mu^+]^{-1})}\Lambda$ restricted to
the tangent subspace
${\mathscr T}_{[\mu^-]\ast [\bar \mu^+]^{-1}}T_B^- \cong \mathrm{BMOA}(\mathbb H^+)$
can be regarded as the identity map. 
Hence, $d_{([\mu^+],[\mu^-])}\, \Lambda({\mathscr T}_{[\mu^-]}T_B^-)=C_{h^+}(\mathrm{BMOA}(\mathbb H^+))$.
The other equation is similarly proved.
\end{proof}

The surjectivity of the derivative of $\Lambda$ is reduced to the following condition.

\begin{lemma}\label{imaginary}
If the real subspace $i\,{\rm Re}\,\mathrm{BMO}(\mathbb R)$ is contained in
${\rm Ran}\,d_{([\mu^+],[\mu^-])}\, \Lambda$, the image of the derivative of $\Lambda$
at $([\mu^+],[\mu^-]) \in \widetilde T_C$, then $d_{([\mu^+],[\mu^-])}\, \Lambda$ is surjective. Similarly,
if the real subspace ${\rm Re}\,\mathrm{BMO}(\mathbb R)$ is contained in
${\rm Ran}\,d_{([\mu^+],[\mu^-])}\, \Lambda$, then $d_{([\mu^+],[\mu^-])}\, \Lambda$ is surjective.
\end{lemma}

\begin{proof}
For the former statement,
it suffices to show that ${\rm Re}\,\mathrm{BMO}(\mathbb R) \subset {\rm Ran}\,d_{([\mu^+],[\mu^-])}\, \Lambda$.
We note that
Lemma \ref{factor} implies that 
$C_{h^+}(\mathrm{BMOA}(\mathbb H^+)) \subset {\rm Ran}\,d_{([\mu^+],[\mu^-])}\, \Lambda$.

We take any $\phi \in {\rm Re}\,\mathrm{BMO}(\mathbb R)$. Since $C_{h^+}$ maps ${\rm Re}\,\mathrm{BMO}(\mathbb R)$
isomorphically onto itself, $C_{h^+}^{-1}(\phi)$ also belongs to ${\rm Re}\,\mathrm{BMO}(\mathbb R)$.
Its Szeg\"o projection
\begin{equation}\label{PC}
P^+(C_{h^+}^{-1}(\phi))=\frac{1}{2}C_{h^+}^{-1}(\phi)+i \frac{1}{2}\mathcal H \circ C_{h^+}^{-1}(\phi)
\end{equation}
is in $\mathrm{BMOA}(\mathbb H^+)$. This implies that
$$
\phi +i C_{h^+} \circ \mathcal H \circ C_{h^+}^{-1}(\phi) \in 
{\rm Ran}\,d_{([\mu^+],[\mu^-])}\, \Lambda
$$
by the application of $2C_{h^+}$ to \eqref{PC}. Here, 
$i C_{h^+} \circ \mathcal H \circ C_{h^+}^{-1}(\phi) \in i\,{\rm Re}\,\mathrm{BMO}(\mathbb R)$ also belongs to
${\rm Ran}\,d_{([\mu^+],[\mu^-])}\, \Lambda$ by the assumption. Therefore, we have that
$\phi \in {\rm Ran}\,d_{([\mu^+],[\mu^-])}\, \Lambda$.

For the latter statement,
it suffices to show that $i\,{\rm Re}\,\mathrm{BMO}(\mathbb R) \subset {\rm Ran}\,d_{([\mu^+],[\mu^-])}\, \Lambda$.
We take any $i\phi \in i\,{\rm Re}\,\mathrm{BMO}(\mathbb R)$. By \eqref{PC}, we have that
$$
iP^+(C_{h^+}^{-1}(\phi))=\frac{i}{2}C_{h^+}^{-1}(\phi)- \frac{1}{2}\mathcal H \circ C_{h^+}^{-1}(\phi)
$$
is in $\mathrm{BMOA}(\mathbb H^+)$, and hence
$$
i\phi - C_{h^+} \circ \mathcal H \circ C_{h^+}^{-1}(\phi) \in 
{\rm Ran}\,d_{([\mu^+],[\mu^-])}\, \Lambda.
$$
Here, $C_{h^+} \circ \mathcal H \circ C_{h^+}^{-1}(\phi) \in {\rm Re}\,\mathrm{BMO}(\mathbb R)$ also belongs to
${\rm Ran}\,d_{([\mu^+],[\mu^-])}\, \Lambda$ by the assumption, and thus we have 
$i\phi \in {\rm Ran}\,d_{([\mu^+],[\mu^-])}\, \Lambda$.
\end{proof}

The following is the crucial claim for establishing our argument. If $\Lambda([\mu_0^+], [\mu_0^-]) = \log \gamma'$ lies in $i\,{\rm Re}\,\mathrm{BMO}(\mathbb{R})$, then $\gamma(x) = \int_0^x \gamma'(t)\, dt$ is the arc-length parametrization of a chord-arc curve, since $|\gamma'(t)| = 1$. In this context, there are several important studies in real analysis regarding the deformation of chord-arc curves, which are surveyed in Semmes \cite{SeB}. 
We extend these results to:

\begin{lemma}\label{new}
If $\Lambda([\mu_0^+],[\mu_0^-])$ is in the subspace $i\,{\rm Re}\,\mathrm{BMO}(\mathbb R)$, 
then $i\,{\rm Re}\,\mathrm{BMO}(\mathbb R)$ is in ${\rm Ran}\, d_{([\mu^+],[\mu^-])} \Lambda$,
the image of the derivative of $\Lambda$ at $([\mu^+],[\mu^-])$.
\end{lemma}

\begin{proof}
This can be seen from the work on chord-arc curves
by Semmes \cite[p.254]{Se}. In fact, for every $\phi_0=\Lambda([\mu_0^+],[\mu_0^-]) \in i\,{\rm Re}\,\mathrm{BMO}(\mathbb R)$,
there is a neighborhood $U_0$ of $\phi_0$ in $\mathrm{BMO}(\mathbb R)$ and a holomorphic map 
$\tau:U_0 \to M_B(\mathbb H^+) \times M_B(\mathbb H^-)$ such that $\Lambda \circ (\pi^+ \times \pi^-) \circ \tau$
is the identity on $U_0$. To prove this claim, we first see that $\tau$ is defined to be bounded 
on some neighborhood $U_0$
by the quasiconformal extension given in \cite[Proposition 4.13]{Se}. Then, 
by Shen and Wei \cite[Formulas (6.7),(6.27)]{SW}, the complex dilatation $\tau(\phi)$
for $\phi \in U_0$ is explicitly represented, and $\tau(\phi)(z)$ is G\^ateaux holomorphic on $U_0$
for each fixed $z \in \mathbb H^+ \cup \mathbb H^-$. Under these conditions, we can conclude that
$\tau$ is holomorphic on $U_0$ by WM \cite[Lemma 6.1]{WM-6}.
\end{proof}

\begin{proof}[Proof of Theorem \ref{main2}]
Proposition \ref{inj} asserts that $\Lambda$ is a holomorphic injection. 
To show that $\Lambda$ is biholomorphic, we prove that the derivative $d_{([\mu^+],[\mu^-])}\, \Lambda$ 
at every $([\mu^+],[\mu^-]) \in \widetilde T_C$ is
surjective onto $\mathrm{BMO}(\mathbb R)$. Then, by the inverse mapping theorem (see Chae \cite[\S 7.18]{Ch}),
we obtain the required claim.

Let $\phi=\Lambda([\mu^+],[\mu^-]) \in \mathrm{BMO}^*(\mathbb R)$. We can find 
$\phi_0 \in i\,{\rm Re}\,\mathrm{BMO}(\mathbb R) \cap \Lambda(\widetilde T_C)$ and $[\nu] \in T_B$
such that $Q_{h}(\phi_0)=\phi$ with $h=h(\nu)$. Indeed, we take $[\nu] \in T_B$ such that 
$\Lambda([\nu],[\bar \nu]) =\log h'={\rm Re}\,\phi$
by Theorem \ref{converse}, and set $\phi_0=iC_h^{-1}({\rm Im}\,\phi)$. Then,
$$
Q_{h}(\phi_0)=C_h(\phi_0)+\log h'=i\,{\rm Im}\,\phi+{\rm Re}\,\phi=\phi.
$$

We consider the derivative of $\Lambda$ at $([\mu^+_0],[\mu^-_0])=R_{[\nu]}^{-1}([\mu^+],[\mu^-])$,
where $\Lambda([\mu^+_0],[\mu^-_0])=Q_{h}^{-1} \circ \Lambda([\mu^+],[\mu^-])=\phi_0$.
By Lemma \ref{new}, we have $i\,{\rm Re}\,\mathrm{BMO}(\mathbb R) \subset {\rm Ran}\,d_{([\mu_0^+],[\mu_0^-])}\, \Lambda$.
Under this condition,
Lemma \ref{imaginary} yields that $d_{([\mu_0^+],[\mu_0^-])}\, \Lambda$ is surjective.
Since 
$$
d_{([\mu^+],[\mu^-])}\,\Lambda=d_{\phi_0}Q_{h} \circ d_{([\mu_0^+],[\mu_0^-])}\,\Lambda 
\circ d_{([\mu^+],[\mu^-])}R_{[\nu]}^{-1},
$$
this is also surjective. 
\end{proof}

The proof of Theorem \ref{main2} shows that 
the derivative 
$$
d_{([\mu^+],[\mu^-])}\Lambda: {\mathscr T}_{([\mu^+],[\mu^-])}\,\widetilde T_C \to \mathrm{BMO}(\mathbb R)
$$
of $\Lambda$ at every point $([\mu^+],[\mu^-]) \in \widetilde T_C$ is a surjective isomorphism.
Therefore, in account of
Lemma \ref{factor}, we obtain the topological direct sum decomposition of $\mathrm{BMO}(\mathbb R)$ 
at $([\mu^+],[\mu^-])$ as
\begin{equation}\label{decomp}
\mathrm{BMO}(\mathbb R)=C_{h^-}(\mathrm{BMOA}(\mathbb H^-)) \oplus C_{h^+}(\mathrm{BMOA}(\mathbb H^+)).
\end{equation} 
According to this decomposition, we define
the bounded linear projections as
\begin{equation}\label{tangent2}
P^\pm_{([\mu^+],[\mu^-])}:\mathrm{BMO}(\mathbb R) \to C_{h^\pm}(\mathrm{BMOA}(\mathbb H^\pm)).
\end{equation}

\begin{lemma}\label{conjugate}
$P^\pm_{([\mu^+],[\mu^-])}=d_{([\mu^+],[\mu^-])}\,\Lambda\circ J^\pm\circ (d_{([\mu^+],[\mu^-])}\,\Lambda)^{-1}$.
\end{lemma} 

\begin{proof}
By Lemma \ref{factor}, 
we see that the derivative $d_{([\mu^+],[\mu^-])}\, \Lambda$ preserves each factor of 
the direct sum decompositions \eqref{T-tangent} and \eqref{decomp}.
Then, the projections $J^+$ and $P^+_{([\mu^+],[\mu^-])}$ to the second factors are conjugated 
by $d_{([\mu^+],[\mu^-])}\, \Lambda$, and so are the projections
$J^-$ and $P^-_{([\mu^+],[\mu^-])}$ to the first factors.
\end{proof}

Finally, as an application of Theorem \ref{main2}, we obtain a result about
the real-analytic structure of the BMO Teich\-m\"ul\-ler space $T_B$.
We restrict the biholomorphic homeomorphism $\Lambda$ to the real-analytic submanifold ${\rm Sym}\,(T_B^+ \times T_B^-)$, 
as in the setting of Theorem \ref{main1}. 
By composing $\Lambda$ with the canonical real-analytical embedding 
$\iota:T_B \to {\rm Sym}\,(T_B^+ \times T_B^-)$, we have the following. 
This has appeared in WM \cite[Corollary 6.2]{WM-0}.

\begin{corollary}\label{structure}
The map $\Lambda \circ \iota:T_B \to {\rm Re}\,\mathrm{BMO}^*(\mathbb R)$
given by $h \mapsto \log h'$
is a real-analytic diffeomorphism. Hence, the BMO Teich\-m\"ul\-ler space $T_B$ is real-analytically equivalent to 
${\rm Re}\,\mathrm{BMO}^*(\mathbb R)$. 
\end{corollary}

\section{The Cauchy transform on chord-arc curves}\label{section8}

In the course of proving
the biholomorphic correspondence $\Lambda:\widetilde T_C \to \mathrm{BMO}^*(\mathbb R)$ 
from BMO embeddings with chord-arc image in Bers coordinates to the BMO space,
we investigated the derivative $d_{([\mu^+],[\mu^-])}\Lambda$ at $([\mu^+],[\mu^-]) \in \widetilde T_C$.
In this section, we show that the Cauchy transform and the Cauchy projection of BMO functions on 
the chord-arc curve $\Gamma=\gamma(\mathbb R)$ for $\gamma=\gamma([\mu^+],[\mu^-])$ 
can be represented by $d_{([\mu^+],[\mu^-])}\Lambda$.
In particular, this proves the Calder\'on theorem in real analysis
for BMO functions on chord-arc curves.

We first introduce the Cauchy transform and the Cauchy projection by
considering the Cauchy integral on a chord-arc curve $\Gamma = \gamma(\mathbb R)$.
We define the Banach space of BMO functions on $\Gamma$ by
the push-forward of $\mathrm{BMO}(\mathbb R)$ by $\gamma$ and identify this pair. Namely, 
$$
\mathrm{BMO}(\gamma(\mathbb R))=\{\gamma_*\phi=\phi \circ \gamma^{-1} \mid \phi \in \mathrm{BMO}(\mathbb R)\}
$$
with norm $\Vert \gamma_*\phi \Vert_{\mathrm{BMO}(\gamma)}=\Vert \phi \Vert_{\mathrm{BMO}}$. 

\begin{remark}
Usually, a function space on a locally rectifiable curve $\Gamma$ is defined
by using its arc-length parametrization.
In our case, for a BMO embedding $\gamma_0:\mathbb R \to \mathbb C$ 
such that $\gamma_0$ gives an arc-length parametrization of $\Gamma$,
we may set $\mathrm{BMO}(\gamma_0(\mathbb R))$ by the above definition. However, the difference between
$\gamma$ and $\gamma_0$ is given by the composition operator $C_{h}$
for $h \in \rm SQS$
(as explained in the next section in detail), and it can be controlled well. Hence, we adopt
the representation $\mathrm{BMO}(\gamma(\mathbb R))$ because it simplifies the arguments 
for considering the dependence of the operators acting on it
when the embedding $\gamma$ varies.
\end{remark}

\begin{definition}\label{integral} 
The {\it Cauchy transform} of  
$\psi \in \mathrm{BMO}(\gamma(\mathbb R))$ on a chord-arc curve $\Gamma = \gamma(\mathbb R)$ 
(oriented by $\mathbb R$ in the increasing direction) is defined by the singular integral
$$
({\mathcal H}_\Gamma \psi)(\xi)={\rm p.v.} \frac{1}{\pi} \int_{\Gamma} 
\left(\frac{\psi(z)}{\xi-z}-\frac{\psi(z)}{\zeta_0^\pm-z}\right)dz
\quad(\xi \in \Gamma),
$$
where $dz=\gamma'(t)dt$, and $\zeta_0^+$ and $\zeta_0^-$ are arbitrary points in the left and the right domains $\Omega^+$ and $\Omega^-$
bounded by $\Gamma$, respectively.
The Cauchy integrals of $\psi$ on $\Gamma$ are defined by
$$
({P}^\pm_\Gamma \psi)(\zeta)=\frac{-1}{2\pi i} \int_{\Gamma} 
\left(\frac{\psi(z)}{\zeta-z}-\frac{\psi(z)}{\zeta_0^\pm-z}\right)dz
\quad (\zeta \in \Omega^\pm),
$$
which are holomorphic functions on $\Omega^\pm$. Here, the integration over $\Gamma$ is taken in the inverse direction
when $\zeta \in \Omega^-$.
\end{definition}

The pointwise convergence of the Cauchy transform and the Cauchy integrals for $\psi \in \mathrm{BMO}(\gamma(\mathbb R))$ 
are guaranteed by this definition of using
the regularized kernel. By the Privalov theorem (see Goluzin \cite[p.431]{Go}), 
if the Cauchy transform $({\mathcal H}_\Gamma \psi)(\xi)$ exists
a.e. on $\Gamma$, then the Cauchy integrals $({P}^\pm_\Gamma \psi)(\zeta)$ have non-tangential limits a.e. on $\Gamma$,
and vice versa.
The boundary extensions of 
${P}^\pm_\Gamma \psi$ to $\Gamma$ 
are also denoted by the same symbol and called the {\it Cauchy projections} of $\psi$.

The {\it Plemelj} {\it formula} for the Riemann--Hilbert problem asserts
the following relation between the Cauchy transform and the Cauchy projections.
This is a generalization of the relation between the Hilbert transform and the Szeg\"o projections.
We remark that the sign of $P^-_\Gamma$ is opposite to the usual one due to the orientation of $\Gamma$.

\begin{proposition}\label{Plemelj}
For a function $\psi \in \mathrm{BMO}(\gamma(\mathbb R))$ 
on a chord-arc curve $\Gamma=\gamma(\mathbb R)$, the Cauchy transform ${\mathcal H}_\Gamma$ 
and the Cauchy projections $P^\pm_\Gamma$ satisfy
\begin{equation*}
P^+_\Gamma \psi = \frac{1}{2}(\psi + i{\mathcal H}_\Gamma \psi),\quad
P^-_\Gamma \psi = \frac{1}{2}(\psi - i{\mathcal H}_\Gamma \psi).
\end{equation*}
In other words,
\begin{equation*}
\psi = P^+_\Gamma \psi + P^-_\Gamma \psi,\quad
i{\mathcal H}_\Gamma \psi = P^+_\Gamma \psi - P^-_\Gamma \psi
\end{equation*}
holds.
\end{proposition}

Next, we consider the space of holomorphic functions to which the Cauchy integrals ${P}^\pm_\Gamma \psi$ belong.
Let $F^\pm:\mathbb H^\pm \to \Omega^\pm$ be the normalized Riemann mappings,
where $\Omega^\pm$ are the complementary domains of $\Gamma$ in $\mathbb C$.
Then, we define the Banach space of BMOA functions on $\Omega^\pm$ 
by the push-forward of $\mathrm{BMOA}(\mathbb H^\pm)$ by $F^\pm$ and 
identify these pairs. Namely,
$$
\mathrm{BMOA}(\Omega^\pm)=\{(F^\pm)_* \Phi^\pm=\Phi^\pm \circ (F^\pm)^{-1} \mid \Phi^\pm \in \mathrm{BMOA}(\mathbb H^\pm)\}
$$
with norm $\Vert (F^\pm)_* \Phi^\pm \Vert_{\mathrm{BMOA}(\Omega^\pm)}=\Vert \Phi^\pm \Vert_{\mathrm{BMOA}}$.

For $(F^\pm)_* \Phi^\pm \in \mathrm{BMOA}(\Omega^\pm)$, 
their boundary extensions to $\Gamma$ are $E(\Phi^\pm) \circ (F^\pm)^{-1}$,
where the Riemann mappings $F^\pm$ are assumed to extend to homeomorphisms of $\mathbb R$ onto $\Gamma$. 
We define
the boundary extension operators $E_\Gamma^\pm$ to $\Gamma$ by
$$
E_\Gamma^\pm((F^\pm)_* \Phi^\pm)=E(\Phi^\pm) \circ (F^\pm)^{-1}.
$$

\begin{proposition}\label{BMOtrace}
It holds that $E_\Gamma^\pm(\mathrm{BMOA}(\Omega^\pm)) \subset \mathrm{BMO}(\gamma(\mathbb R))$.  
Moreover, these trace operators $E_\Gamma^\pm$ are Banach isomorphisms onto their images, where their operator norms are
estimated in terms of the norms of the composition operators $C_{h^\pm}$
for $h^\pm = h(\mu^\pm) \in \rm SQS$, respectively.
\end{proposition}

\begin{proof}
The norm on $\mathrm{BMOA}(\Omega^\pm)$ is induced from $\mathrm{BMOA}(\mathbb H^\pm)$
by $F^\pm$ whereas the norm of $\mathrm{BMO}(\gamma(\mathbb R))$ is induced from $\mathrm{BMO}(\mathbb R)$ by $\gamma$.
Because $\gamma=F^\pm \circ h^\pm$, 
the differences between $E^\pm$ and $E_\Gamma^\pm$
are caused by the composition operators $C_{h^\pm}$, respectively.
\end{proof}

By this proposition, we see that $E_\Gamma^\pm(\mathrm{BMOA}(\Omega^\pm))$ 
are closed subspaces of $\mathrm{BMO}(\gamma(\mathbb R))$. Moreover, under the Banach isomorphism $E_\Gamma^\pm$,
we may identify $E_\Gamma^\pm(\mathrm{BMOA}(\Omega^\pm))$ with $\mathrm{BMOA}(\Omega^\pm)$. Hence,
we regard $\mathrm{BMOA}(\Omega^\pm)$ as closed subspaces of
$\mathrm{BMO}(\gamma(\mathbb R))$ without noticing $E_\Gamma^\pm$ hereafter.

For $([\mu^+],[\mu^-]) \in \widetilde T_C$, we have obtained 
the images of the tangent subspaces ${\mathscr T}_{[\mu^+]} T_B^+$ and
${\mathscr T}_{[\mu^-]} T_B^-$ under the surjective derivative $d_{([\mu^+],[\mu^-])}\Lambda$, 
which correspond to the topological direct sum decomposition
\eqref{decomp}.
Then, every $\phi \in \mathrm{BMO}(\mathbb R)$ is uniquely
represented by $\phi=\phi^++\phi^-$ for $\phi^+ \in C_{h^+}(\mathrm{BMOA}(\mathbb H^+))$ and
$\phi^- \in C_{h^-}(\mathrm{BMOA}(\mathbb H^-))$. These bounded linear projections are denoted by
$\phi^\pm=P^\pm_{([\mu^+],[\mu^-])}(\phi)$.

A novel observation in our arguments is that the Cauchy projection can be related to
the derivative of $\Lambda$. The following result is crucial in this sense and provide a new method
for the study of the Cauchy projection on a chord-arc curve.

\begin{theorem}\label{Cauchy}
In the function space $\mathrm{BMO}(\gamma(\mathbb R))$ on a chord-arc curve $\Gamma=\gamma(\mathbb R)$, 
the Cauchy projections $P^\pm_\Gamma$ satisfy
$$
P^\pm_\Gamma = \gamma_* \circ P^\pm_{([\mu^+],[\mu^-])} \circ \gamma_*^{-1}.
$$
In particular, $P^\pm_\Gamma$ maps $\mathrm{BMO}(\gamma(\mathbb R))$ onto $\mathrm{BMOA}(\Omega^\pm)$, 
they are bounded linear operators, 
and their operator norms are estimated in terms of $\Vert C_{h^\pm}\Vert$.
\end{theorem}

\begin{proof}
Let $\phi=\gamma_*^{-1}(\psi)=\psi \circ \gamma$ for 
$\psi \in \mathrm{BMO}(\gamma(\mathbb R))=\gamma_*(\mathrm{BMO}(\mathbb R))$. Then,
$\phi \in \mathrm{BMO}(\mathbb R)$ and $ P^\pm_{([\mu^+],[\mu^-])} \circ \gamma_*^{-1}(\psi)=\phi^\pm
\in C_{h^\pm}(\mathrm{BMOA}(\mathbb H^\pm))$. Therefore,
$$
\gamma_* \circ P^\pm_{([\mu^+],[\mu^-])} \circ \gamma_*^{-1}(\psi) \in
\mathrm{BMOA}(\Omega^\pm)
$$
since $\gamma=F^\pm \circ h^\pm$ and $\mathrm{BMOA}(\Omega^\pm)=F^\pm_* (\mathrm{BMOA}(\mathbb H^\pm))$
for the Riemann mappings $F^\pm:\mathbb H^\pm \to \Omega^\pm$.
By the definition of the projections $P^\pm_{([\mu^+],[\mu^-])}$, we have
\begin{equation}\label{decomp1}
\gamma_* \circ P^+_{([\mu^+],[\mu^-])} \circ \gamma_*^{-1}(\psi)+\gamma_* 
\circ P^-_{([\mu^+],[\mu^-])} \circ \gamma_*^{-1}(\psi)=\psi.
\end{equation}
Let $\Psi_1$ be a measurable function on $\mathbb C$ defined by BMOA functions 
$\gamma_* \circ P^+_{([\mu^+],[\mu^-])} \circ \gamma_*^{-1}(\psi)$ on $\Omega^+$ and
$-\gamma_* \circ P^-_{([\mu^+],[\mu^-])} \circ \gamma_*^{-1}(\psi)$ on $\Omega^-$. Then,
$\Psi_1$ is locally integrable and is of growth order $\Psi_1(z)=o(|z|)$ as $z \to \infty$.

We can verify that $-2i \bar \partial \Psi_1=\psi dz_\Gamma$ as a distribution
according to the argument in Semmes \cite[p.204]{Se1}. 
Here, $dz_\Gamma$ denotes a continuous linear functional 
defined by
$$
\langle X, dz_\Gamma \rangle=\int_\Gamma X(z)dz
$$
for every test function $X \in C_0^\infty(\mathbb C)$. Then, we have 
\begin{align}
\langle X, \psi dz_\Gamma \rangle=\int_\Gamma X(z)\psi(z)dz=2i \int_{\mathbb C} \bar \partial X(z)\Psi_1(z)dxdy
=-2i\langle X, \bar \partial \Psi_1 \rangle.
\end{align}
The middle equality in the above equations is derived from the Green formula. 

To see this more precisely, we choose a simply connected bounded domain 
$W \subset \mathbb C$ of smooth boundary intersecting $\Gamma$ and containing the support of $X$. Moreover, we choose a decreasing sequence of
neighborhoods $U_n$ of $\Gamma$ with $\bigcap_{n=1}^\infty U_n=\Gamma$ appropriately,
and let $W^\pm_n=(W \cap \Omega^\pm) \setminus U_n$. Then, the Green formula implies that
$$
\int_{\partial W_n^\pm}  X(z)\Psi_1(z)dz=\int_{W_n^\pm}  \bar \partial(X(z)\Psi_1(z))d\bar z \wedge dz
=2i \int_{W_n^\pm}  \bar \partial X(z) \Psi_1(z)dxdy.
$$
Here, the first integrals over $\partial W_n^\pm$ tend to the integrals involving the non-tangential limits of $\Psi_1$
over $\Gamma \cap W$ as $n \to \infty$, because $\Psi_1$ is made of ${\rm BMOA}(\Omega^\pm)$ 
whose restrictions to $W^\pm=W \cap \Omega^\pm$ belong to the Hardy space $H^p(W^\pm)$.
Then, summing up two equations for $\pm$ after passing to these limits, we obtain
the required equation.

The Cauchy projections 
$P^\pm_\Gamma(\psi)$ satisfy the same properties as $\gamma_* \circ P^\pm_{([\mu^+],[\mu^-])} \circ \gamma_*^{-1}(\psi)$.
They are
holomorphic functions on $\Omega^\pm$ such that 
their boundary extensions satisfy $P^+_\Gamma(\psi)+P^-_\Gamma(\psi)=\psi$ by Proposition \ref{Plemelj}. Let $\Psi_2$ be
a measurable function on $\mathbb C$ defined by $P^+_\Gamma(\psi)$ on $\Omega^+$ and
$-P^-_\Gamma(\psi)$ on $\Omega^-$.
Then, $\Psi_2$ is locally integrable and is of growth order $\Psi_2(z)=o(|z|)$ as $z \to \infty$.
These facts are verified in \cite[Lemma 3.2]{Se1}.

Moreover, we have also $-2i \bar \partial \Psi_2=\psi dz_\Gamma$ as a distribution.
Indeed, for every test function $X \in C_0^\infty(\mathbb C)$, the Pompeiu formula (applied at $\dot =$) implies that
\begin{align}
-2i\langle X, \bar \partial \Psi_2 \rangle 
&=2i \int_{\mathbb C} \bar \partial X(\zeta)\Psi_2(\zeta)d\xi d\eta
=\frac{-1}{\pi}\int_{\mathbb C} \bar \partial X(\zeta)\left(\int_{\Gamma} \left(\frac{\psi(z)}{\zeta-z}-\frac{\psi(z)}{\zeta_0^\pm-z}\right)dz\right) d\xi d\eta\\
&=\int_{\Gamma}\left(\frac{-1}{\pi}\int_{\mathbb C}\frac{\bar \partial X(\zeta)}{\zeta-z}d\xi d\eta\right) \psi(z) dz
\,\dot=\int_{\Gamma} X(z) \psi(z)dz=\langle X, \psi dz_\Gamma \rangle.
\end{align}
The exchange of order of the integrations is guaranteed simply by integrability of
$$
|\bar \partial X(\zeta)|\left|\frac{\psi(z)}{\zeta-z}-\frac{\psi(z)}{\zeta_0^\pm-z}\right|
$$
over $(\zeta,z) \in \mathbb C \times \Gamma$.

Since we have seen that $-2i \bar \partial \Psi_1=\psi dz_\Gamma=-2i \bar \partial \Psi_2$,
$\bar \partial$-derivative of the locally integrable function $\Psi_1-\Psi_2$ is $0$ on $\mathbb C$
in the distribution sense. Hence, it is holomorphic on $\mathbb C$ by the Weyl lemma. The growth order 
$\Psi_1(z)-\Psi_2(z)=o(|z|)$ $(z \to \infty)$ forces the entire function $\Psi_1-\Psi_2$ to be a constant.
Thus, we have $P^\pm_\Gamma(\psi)=\gamma_* \circ P^\pm_{([\mu^+],[\mu^-])} \circ \gamma_*^{-1}(\psi)$ on $\Omega^\pm$ up to constants.
\end{proof}

\begin{corollary}\label{BMOauto}
The Cauchy transform ${\mathcal H}_\Gamma$ on a chord-arc curve $\Gamma=\gamma(\mathbb R)$
maps $\mathrm{BMO}(\gamma(\mathbb R))$ to $\mathrm{BMO}(\gamma(\mathbb R))$, which is
a Banach automorphism of $\mathrm{BMO}(\gamma(\mathbb R))$.
\end{corollary}

\begin{proof}
Since ${\mathcal H}_\Gamma=-i(P^+_\Gamma - P^-_\Gamma)$ by Proposition \ref{Plemelj},
the statement follows from Proposition \ref{BMOtrace} and Theorem \ref{Cauchy}.
\end{proof}

\begin{remark}
The boundedness of the Cauchy transform ${\mathcal H}_\Gamma$, as well as 
the Cauchy projections $P^\pm_\Gamma$,
is also verified in Liu and Shen \cite[Theorem 1.1]{LS1} as an application of  
the corresponding $L^p$ estimate in David \cite{Da}.
\end{remark}

Theorem \ref{Cauchy} also leads to {\it the Calder\'on theorem} for chord-arc curves.
See Coifman and Meyer \cite[Section 8]{CM0}, where this is proved in the small norm case.

\begin{corollary}\label{Calderon}
The Cauchy projections $P_\Gamma^\pm$ on a chord-arc curve $\Gamma=\gamma(\mathbb R)$ 
are associated with the topological direct sum decomposition
$$
\mathrm{BMO}(\gamma(\mathbb R)) = \mathrm{BMOA}(\Omega^+) \oplus \mathrm{BMOA}(\Omega^-).
$$
\end{corollary}

\begin{remark}
$(F^\pm_*)^{-1} \circ P^\pm_\Gamma \circ \gamma_*=C_{h^\pm}^{-1} \circ P^\pm_{([\mu^+],[\mu^-])}$ maps $\mathrm{BMO}(\mathbb R)$ onto 
$\mathrm{BMOA}(\mathbb H^\pm)$.
When $[\mu^+]=[0]$ or $[\mu^-]=[0]$, this coincides with what is called the {\it Faber operator}.
See Liu and Shen \cite{LS1, LS3} for related arguments. 
In our setting, we see not only the boundedness of this operator but
also its holomorphic dependence when the embeddings $\gamma$ vary in the Teich\-m\"ul\-ler space as is discussed in the next section.
\end{remark}

\section{Holomorphic dependence of the Cauchy transform}\label{section9}

In this section, we consider the variation of the Cauchy transform ${\mathcal H}_\Gamma$
when $\Gamma=\gamma(\mathbb R)$ moves according to $\gamma=\gamma([\mu^+],[\mu^-])$.
To formulate this problem, we take the conjugate of ${\mathcal H}_\Gamma$ so that 
it acts on $\mathrm{BMO}(\mathbb R)$. Namely, for $([\mu^+],[\mu^-]) \in \widetilde T_C$, we set
\begin{equation}\label{HP}
{\mathcal H}_{([\mu^+],[\mu^-])}
=\gamma_*^{-1} \circ {\mathcal H}_\Gamma \circ \gamma_*,
\end{equation}
which is a Banach automorphism of $\mathrm{BMO}(\mathbb R)$ by Corollary \ref{BMOauto}.
More explicitly,
$$
{\mathcal H}_{([\mu^+],[\mu^-])}(\phi)(x)=
{\rm p.v.} \frac{1}{\pi} \int_{-\infty}^{\infty} \frac{\phi(t)}{\gamma(x)-\gamma(t)}\gamma'(t)dt
\quad(x \in \mathbb R)
$$
for $\phi \in \mathrm{BMO}(\mathbb R)$.

By Proposition \ref{Plemelj}, the relationship between the Cauchy transform ${\mathcal H}_\Gamma$
and the Cauchy projections $P_\Gamma^\pm$ is given as
${\mathcal H}_\Gamma=-i(P_\Gamma^+-P_\Gamma^-)$.
Moreover, Theorem \ref{Cauchy} shows $P^\pm_\Gamma=\gamma_* \circ P^\pm_{([\mu^+],[\mu^-])} \circ \gamma_*^{-1}$.
Then, we have
\begin{equation}\label{HP}
{\mathcal H}_{([\mu^+],[\mu^-])}=-i( P^+_{([\mu^+],[\mu^-])} -P^-_{([\mu^+],[\mu^-])}).
\end{equation}

For $([0],[0]) \in \widetilde T_C$, 
${\mathcal H}_{([0],[0])}$ coincides with the Hilbert transform $\mathcal H$. 
For $([\mu],[\bar \mu]) \in {\rm Sym}\,(T_B^+ \times T_B^-)$, 
${\mathcal H}_{([\mu],[\bar \mu])}$
is the conjugate of $\mathcal H$ by the composition operator $C_h$ for $h=h(\mu) \in \rm SQS$. 
Indeed, \eqref{HP} applied to the case of $\Gamma=h(\mathbb R)=\mathbb R$ with 
$\gamma([\mu],[\bar \mu])=h$ yields ${\mathcal H}_{([\mu],[\bar \mu])}=C_h \circ {\mathcal H} \circ C_h^{-1}$.

Let ${\mathcal L}(\mathrm{BMO}(\mathbb R))$ be
the Banach space of all bounded linear operators $\mathrm{BMO}(\mathbb R) \to \mathrm{BMO}(\mathbb R)$
equipped with the operator norm. We consider the map
$$
\eta:\widetilde T_C \to {\mathcal L}(\mathrm{BMO}(\mathbb R))
$$
defined by $([\mu^+],[\mu^-]) \mapsto {\mathcal H}_{([\mu^+],[\mu^-])}$.
An advantage of our arguments developed in this paper is to be ready to obtain the
holomorphic dependence of $\eta$, which has not appeared explicitly in the literature as a statement.

\begin{theorem}\label{conjugateH}
$\eta:\widetilde T_C \to {\mathcal L}(\mathrm{BMO}(\mathbb R))$ is holomorphic.
\end{theorem}

\begin{proof}
Due to formula \eqref{HP}, it suffices to show that $P^\pm_{([\mu^+],[\mu^-])}$ 
depend holomorphically on $([\mu^+],[\mu^-]) \in \widetilde T_C$.
By Lemma \ref{conjugate}, we have
\begin{align}
P^\pm_{([\mu^+],[\mu^-])}&=d_{([\mu^+],[\mu^-])}\,\Lambda\circ J^\pm\circ (d_{([\mu^+],[\mu^-])}\,\Lambda)^{-1}
=d_{([\mu^+],[\mu^-])}\,\Lambda\circ J^\pm\circ d_{\Lambda([\mu^+],[\mu^-])}\,\Lambda^{-1}.
\end{align}
Since $\Lambda$ is biholomorphic, these operators depend holomorphically on $([\mu^+],[\mu^-]) \in \widetilde T_C$.
\end{proof}

A related result can be found in Coifman and Meyer \cite[Th\'eor\`eme 1]{CM0}, 
though without involving the biholomorphic map $\Lambda$, where it is demonstrated using real analytic methods  
that $\eta|_{{\rm Sym}\,(T_B^+ \times T_B^-)}$
is real-analytic. In other words, it is shown that ${\mathcal H}_{([\mu],[\bar \mu])}=C_h \circ {\mathcal H} \circ C_h^{-1}$
depends real-analytically on $h \in {\rm BMO}^*(\mathbb R)$.
This claim immediately follows from Theorem \ref{conjugateH}.

In the sequel, as an application of this claim,
we consider the theorem of Coifman and Meyer \cite[Theorem 1]{CM}. 
A survey of this theorem is in Semmes \cite[Theorem 5]{SeB}.

We define $Z = i\,\mathrm{BMO}(\mathbb R) \cap \Lambda(\widetilde T_C)$ as the real-analytic submanifold of 
$\Lambda(\widetilde T_C)$ consisting of purely imaginary-valued BMO functions which is an open subset of the real Banach subspace
$i\,\mathrm{BMO}(\mathbb R)$ as used in Lemmas \ref{imaginary} and \ref{new}. We do not know whether $Z$ is connected or not.
For $\psi \in Z$, 
$$
g_0(x) = \int_0^x \exp\psi(t)dt \quad (x \in \mathbb R)
$$ 
serves as the {\it arc-length parameterization} of the chord-arc curve $g_0(\mathbb R)$.

In general, for any BMO embedding with chord-arc image
$$
g(x) = \int_0^x \exp \varphi(t)dt \quad (x \in \mathbb R)
$$ 
determined by $\varphi \in \Lambda(\widetilde T_C) \subset {\rm BMO}^*(\mathbb R)$,
we take a strongly quasisymmetric homeomorphism 
$$
h(x) = \int_0^x \exp({\rm Re}\,\varphi(t))dt
$$ 
and set $\psi = i\, {\rm Im}\,\varphi \circ h^{-1} \in Z$. 
Then, 
the arc-length parameter $g_0$ of the chord-arc curve $g(\mathbb R)$
determined by $\psi$ allows $g$ to be expressed as the {\it reparametrization} of $g_0$ by $h$, that is,
$g=g_0 \circ h$.

We define $Y=\mathrm{BMOA}(\mathbb H^+) \cap \Lambda(\widetilde T_C)$ as the complex submanifold of $\Lambda(\widetilde T_C)$ consisting of BMOA functions on $\mathbb H^+$. We do not know whether $Y$ is connected or not.
For $\varphi \in Y$, the corresponding BMO embedding with chord-arc image
$$
f(x) = \int_0^x \exp \varphi(t)dt \quad (x \in \mathbb R),
$$ 
when applied as above, can be expressed as the reparametrization of the arc-length parameter $g_0$ 
by a strongly quasisymmetric homeomorphism $h$
so that
$f = g_0 \circ h$.
This correspondence between $f$ and $g_0$ is bijective, thus defining a mapping from $Z$ to $Y$.
Similarly, the correspondence from the pair $(f, g_0)$ to their reparametrization $h$ allows 
the definition of a mapping from $Z$ to ${\rm Re}\,\mathrm{BMO}^*(\mathbb R)$.

To investigate these mappings on $\widetilde T_C$ through the biholomorphic homeomorphism $\Lambda$, we define
\begin{align*}
\rho&: \widetilde T_C \to \{[0]\} \times T_C^-, \quad ([\mu^+],[\mu^-]) \mapsto ([0], [\mu^-] \ast [\overline{\mu^+}]^{-1});\\
\delta&: \widetilde T_C \to {\rm Sym}\,(T_B^+ \times T_B^-), \quad ([\mu^+],[\mu^-]) \mapsto ([\mu^+],[\overline{\mu^+}]).
\end{align*}
Here, $\delta$ is the projection to the symmetric axis, which is real-analytic. 
The unique decomposition 
$$
([\mu^+],[\mu^-]) = \rho([\mu^+],[\mu^-]) \ast \delta([\mu^+],[\mu^-])
$$
corresponds to the decomposition of a quasisymmetric embedding $g=\gamma([\mu^+],[\mu^-])$ into $g=f \circ h$ in general, where
$f=\gamma(\rho([\mu^+],[\mu^-]))$ is the boundary extension of the conformal homeo\-morphism of $\mathbb H^+$ to $\mathbb R$, and
$h=\gamma(\delta([\mu^+],[\mu^-]))$ is a quasisymmetric homeomorphism of $\mathbb R$.

We transform the two maps defined on the submanifold $Z \subset \Lambda(\widetilde T_C)$
into those on $\widetilde T_C$
via $\Lambda$.
The map $Z \to Y$ corresponds to
$$
\rho_0 = \rho|_{\Lambda^{-1}(Z)}: \Lambda^{-1}(Z) \to \{[0]\} \times T_C^-=\Lambda^{-1}(Y),
$$
and the map $Z \to {\rm Re}\,\mathrm{BMO}^*(\mathbb R)$ corresponds to
$$
\delta_0 = \delta|_{\Lambda^{-1}(Z)}: 
\Lambda^{-1}(Z) \to {\rm Sym}\,(T_B^+ \times T_B^-)=\Lambda^{-1}({\rm Re}\,\mathrm{BMO}^*(\mathbb R)).
$$
We set $\widetilde Z=\Lambda^{-1}(Z)$ and $\widetilde Y=\Lambda^{-1}(Y)$.

We formulate the result on the map $\delta_0$ shown in \cite[Theorem 1]{CM} as follows. 
The proof follows from what has been demonstrated earlier. 

\begin{theorem}
$\delta_0: \widetilde Z \to {\rm Sym}\,(T_B^+ \times T_B^-)$ is a real-analytic diffeomorphism onto its image.
\end{theorem}

\begin{proof}
Since $\delta_0$ is a real-analytic, 
it suffices to show that it is injective and its inverse $\delta_0^{-1}$ is also real-analytic. 
As before, we consider the conjugate
$$
\Lambda \circ \delta_0 \circ \Lambda^{-1}:Z \to {\rm Re}\,\mathrm{BMO}^*(\mathbb R).
$$

Let $g_0(x)=\int_0^x \psi(t)dt$ for
$\psi \in Z$ and 
$h(x)=\int_0^x \phi(t)dt$ for
$\phi=\Lambda \circ \delta_0 \circ \Lambda^{-1}(\psi) \in {\rm Re}\,\mathrm{BMO}^*(\mathbb R)$.
Then, $f=g_0 \circ h^{-1}$ is a Riemann mapping parametrization of the chord-arc curve and $\log f' \in Y$.
Taking the logarithm of the derivative of $g_0=f \circ h$, we have
\begin{equation}\label{logder}
\log g_0'= \log f' \circ h+\log h'.
\end{equation}
Since $\log g_0'=\psi$ is purely imaginary and $\log h'=\phi$ is real, the real and imaginary parts of this equation
become
\begin{equation}\label{ReIm}
0={\rm Re} \log f' \circ h+\log h' \quad {\rm and} \quad -i\psi={\rm Im} \log f' \circ h.
\end{equation}
Moreover, since $\log f'$ is the boundary extension of the holomorphic function $\log F'$ for the Riemann mapping $F$
on $\mathbb H^+$,
${\rm Re}\log f'$ and ${\rm Im}\log f'$ are related by the Hilbert transform $\mathcal H$ on $\mathbb R$:
\begin{equation}\label{Hilbert}
{\rm Im}\log f'={\mathcal H}({\rm Re}\log f').
\end{equation}
Then, the combination of \eqref{ReIm} and \eqref{Hilbert} yields that
\begin{equation}\label{injective}
-C_h\circ {\mathcal H} \circ C_h^{-1}(\log h')=-i\psi.
\end{equation}

This shows that $\psi=\log g_0'$ is determined by $\phi=\log h'$ and 
thus $\Lambda \circ \delta_0 \circ \Lambda^{-1}:\psi \mapsto \phi$ is injective.
Equation \eqref{injective} also gives
$$
\Lambda \circ \delta_0^{-1} \circ \Lambda^{-1}(\phi)=-iC_h\circ {\mathcal H} \circ C_h^{-1}(\phi)
=-i{\mathcal H}_{\Lambda^{-1}(\phi)} \phi.
$$
Then, because the Cauchy transform ${\mathcal H}_{\Lambda^{-1}(\phi)}$
depends real-analytically on $\phi \in {\rm Re}\,\mathrm{BMO}^*(\mathbb R)$
by Theorem \ref{conjugateH}, we see that $\Lambda \circ \delta^{-1}_0 \circ \Lambda^{-1}$ is real-analytic.
\end{proof}

Finally, we mention discontinuity of the map $\rho_0$ in brief.
We consider the problem of continuous dependence of parameters of chord-arc curves
given by the Riemann mappings. 
For any chord-arc curve $\Gamma$, 
the normalized Riemann mapping from $\mathbb H^+$ to the domain enclosed by $\Gamma$ 
defines a BMO embedding $\gamma:\mathbb R \to \Gamma$.
The set of all such BMO embeddings is identified with
$\widetilde Y =  \{[0]\} \times T_C^-$, 
which is a complex analytic submanifold of $\widetilde T_C$.
The map $Z \to Y$ and $\rho_0:\widetilde Z \to \widetilde Y$ give the correspondence of
these Riemann mapping parametrizations to the chord-arc curves with arc length parametrizations.
The problem of asking whether $\rho_0$ is continuous or not is essentially raised in 
Katznelson, Nag and Sullivan \cite[p.303]{KNS}.
It is answered in WM \cite[Theorem 8.3]{WM-0} as follows.

\begin{theorem}\label{discontinuous}
$\rho_0:\widetilde Z \to \widetilde Y$ is not continuous.
\end{theorem}

The proof of this result requires an application of the property that 
$T_B \cong {\rm SQS}$ does not form a topological group (see \cite[Proposition 8.1]{WM-0}). From this property,
it follows that $\rho$ is not continuous on $\widetilde T_C$. 
However, the discontinuity of $\rho_0=\rho|_{\widetilde Z}$ is a stronger claim than this
and is derived from observing 
the local behavior of $\delta_0: \widetilde Z \to {\rm Sym}\,(T_B^+ \times T_B^-)$ near the origin. 
The fact that $\delta_0$ is a homeomorphism onto the image
allows for some degree of freedom in choosing elements in 
${\widetilde Z}$ via ${\rm Sym}\,(T_B^+ \times T_B^-)$ 
and enables the construction of a specific sequence of elements 
demonstrating the discontinuity of $\rho_0$.

\section{VMO and asymptotically smooth embeddings}\label{section10}

The final section is devoted to an exposition of the results obtained analogously for the VMO Teichm\"uller space.

A Carleson measure $\lambda$ on $\mathbb{H}$ is said to be {\it vanishing} if
$$
\lim_{|I| \to 0} \frac{\lambda(I \times (0,|I|))}{|I|} = 0,
$$
where $I$ is a bounded interval in $\mathbb{R}$.  
Let $L_V(\mathbb H)$ denote the subspace of $L_B(\mathbb H)$
consisting of all elements $\mu$
such that $\lambda_\mu = |\mu(z)|^2 dx\,dy/y$ is a vanishing Carleson measure on $\mathbb H$.
We observe that $L_V(\mathbb H)$ is closed in $L_B(\mathbb H)$.
Moreover, we define the corresponding space of Beltrami coefficients
as $M_V(\mathbb H) =  M(\mathbb H) \cap L_V(\mathbb H)$.
For $M_V(\mathbb H) \subset M(\mathbb H)$, we define the {\it VMO Teichm\"uller space} 
$T_V$ as $\pi(M_V(\mathbb H))$. This is closed in $T_B$.

Let $A_V(\mathbb H)$ denote the subspace of $A_B(\mathbb H)$ consisting of all holomorphic functions $\Psi$
such that $\lambda^{(2)}_\Psi$ is a vanishing Carleson measure on $\mathbb H$.
This is a closed subspace of $A_B(\mathbb H)$.
Similarly, ${\rm VMOA}(\mathbb H)$ is defined as the closed subspace of ${\rm BMOA}(\mathbb H)$
consisting of all holomorphic functions $\Phi$ such that
$\lambda^{(1)}_\Phi$ is a vanishing Carleson measure on $\mathbb H$. 

We apply the Schwarzian and pre-Schwarzian derivative maps $S$ and $L$ investigated in Propositions \ref{S-holo} and \ref{L-holo} to $M_V(\mathbb H^+)$, where
$S:M_B(\mathbb H^+) \to A_B(\mathbb H^-)$ is a holomorphic map with a local holomorphic right inverse
$\sigma$ at every $\Psi \in S(M_B(\mathbb H^+))$, and 
$L:M_B(\mathbb H^+) \to {\rm BMOA}(\mathbb H^-)$ is a holomorphic map with a holomorphic bijection 
$D:L(M_B(\mathbb H^+)) \to S(M_B(\mathbb H^+))$ satisfying $D \circ L=S$.
We refer to the results from Shen \cite[Theorems 2.1, 2.2]{Sh}. In the unit disk case, the corresponding results are 
in Shen and Wei \cite[Sections 5, 6]{SWei}.

\begin{proposition}
$(1)$
$S$ maps $M_V(\mathbb H^+)$ into $A_V(\mathbb H^-)$, and $\sigma$ maps the local neighborhood in $S(M_V(\mathbb H^+))$
into $A_V(\mathbb H^-)$.
$(2)$
$L$ maps $M_V(\mathbb H^+)$ into $\mathrm{VMOA}(\mathbb H^-)$.
$(3)$
$D$ maps $L(M_V(\mathbb H^+))$ onto $S(M_V(\mathbb H^+))$.
\end{proposition}

These results in particular imply that the Bers embedding $\alpha:T_B \to S(M_B(\mathbb H^+))$ maps $T_V$ onto 
the domain $S(M_V(\mathbb H^+))$ in $A_V(\mathbb H^-)$,
and that the pre-Bers embedding $\beta:T_B \to L(M_B(\mathbb H^+))$ maps $T_V$ onto the domain $L(M_V(\mathbb H^+))$ 
in $\mathrm{VMOA}(\mathbb H^-)$. Here, if we apply Proposition \ref{Vequivalence} below, we have
$$
S(M_V(\mathbb H^+))=S(M_B(\mathbb H^+)) \cap A_V(\mathbb H^-); \quad 
L(M_V(\mathbb H^+))=L(M_B(\mathbb H^+)) \cap \mathrm{VMOA}(\mathbb H^-).
$$
Since $\alpha$ and $\beta$ are biholomorphic homeomorphisms, 
$T_V$ has the complex structure as a closed submanifold of $T_B$, which is biholomorphically equivalent to 
$L(M_V(\mathbb H^+))$ and $S(M_V(\mathbb H^+))$.

Moreover, it is proved in WM \cite[Theorem 1.4]{WM-7} that if $\log F^\mu \in \mathrm{VMOA}(\mathbb H^-)$ 
for $\mu \in M(\mathbb H^+)$, then
$\mu \in M_V(\mathbb H^+)$. 
We note that the strategy of showing this, which is used for $(c) \Rightarrow (a)$ below,
is different from that in Theorem \ref{char}.
From this, we obtain characterizations for a Beltrami coefficient $\mu$
to be in $M_V(\mathbb H^+)$ in terms of $S$ and $L$.

\begin{proposition}\label{Vequivalence}
For a Beltrami coefficient $\mu \in M(\mathbb H^+)$,
the following conditions are equivalent:
$(a)$ $\mu \in M_V(\mathbb H^+)$; $(b)$ $S(\mu) \in A_V(\mathbb H^-)$;
$(c)$ $L(\mu) \in {\rm VMOA}(\mathbb H^-)$. Hence,
\begin{align}
S(M_V(\mathbb H^+))&=S(M(\mathbb H^+)) \cap A_V(\mathbb H^-);\\
L(M_V(\mathbb H^+))&=L(M(\mathbb H^+)) \cap \mathrm{VMOA}(\mathbb H^-).
\end{align}
\end{proposition}

A BMO function $\phi$ on $\mathbb R$ is said to be of {\it vanishing mean oscillation} (VMO) if
$$
\lim_{|I| \to 0}\frac{1}{|I|} \int_I |\phi(x)-\phi_I| dx = 0,
$$
where $I$ is a bounded interval on $\mathbb R$.
The set of all VMO functions on $\mathbb R$ is denoted by ${\rm VMO}(\mathbb R)$.
This is a closed subspace of ${\rm BMO}(\mathbb R)$.
The restriction of the trace operator $E:{\rm BMOA}(\mathbb H) \to {\rm BMO}(\mathbb R)$
to ${\rm VMOA}(\mathbb H)$ is a Banach isomorphism onto its image in ${\rm VMO}(\mathbb R)$.
However, the composition operator $C_h$ for $h \in \rm SQS$ does not necessarily preserve
${\rm VMO}(\mathbb R)$. If $h$ is uniformly continuous in addition, then
$C_h$ preserves ${\rm VMO}(\mathbb R)$ (see WM \cite[Proposition 3.1]{WM-6}).

A strongly quasisymmetric homeomorphism $h$ is said to be {\it strongly symmetric} 
if $\log h' \in {\rm VMO}(\mathbb R)$.
Let ${\rm SS}$ denote the set of all normalized strongly symmetric homeomorphisms.
This set is identified with $T_V$, but it does not form a subgroup of
$\rm SQS$ (see WM \cite[Corollary 5.6]{WM-9}).
For $h \in \rm SS$, 
the map $H$ given by
the variant of the Beurling--Ahlfors extension by the heat kernel is a quasiconformal real-analytic
self-diffeomorphism of $\mathbb H$ whose complex dilatation belongs to $M_V(\mathbb H)$ (see 
WM \cite[Theorem 4.1]{WM-2}).

Because the VMO Teichm\"uller space defined on $\mathbb H$ and $\mathbb R$, as described above, has several defects alongside its desirable properties, we now shift our focus to a smaller VMO Teich\-m\"ul\-ler space 
defined on the unit disk $\mathbb D$ and the unit circle $\mathbb S$.  
Let $\Theta(z) = (z - i) / (z + i)$ be the Cayley transformation of the Riemann sphere, which maps $\mathbb R \cup \{\infty\}$ onto
$\mathbb S$, with $\Theta(\infty) = 1$. The space $\mathrm{VMO}(\mathbb S)$ of VMO functions on $\mathbb S$ 
is defined similarly, and it is a closed subspace of $\mathrm{BMO}(\mathbb S)$.
While BMO functions on $\mathbb S$ and their associated Teichm\"uller spaces are well studied and equivalent 
to those defined on $\mathbb R$—since the Cayley transformation $\Theta$ provides an appropriate correspondence between
all components of the involved spaces—this equivalence does not hold for VMO functions.

We define $\mathrm{VcMO}(\mathbb R)$ as the pull-back of $\mathrm{VMO}(\mathbb S)$ via $\Theta$.
This coincides with the closure of compactly supported VMO functions on $\mathbb R$ with respect to the BMO norm,
and is sometimes denoted by $\mathrm{CMO}(\mathbb R)$. 
In fact, as noted in Coifman and Weiss \cite[p.639]{CW}, this is adopted as the definition of VMO functions on $\mathbb R$.
By definition, $\mathrm{VcMO}(\mathbb R)$ is a closed subspace of $\mathrm{VMO}(\mathbb R) \subset \mathrm{BMO}(\mathbb R)$.
By considering the relationship between $\mathrm{VMO}(\mathbb R)$ and
other spaces involved in the Teichm\"uller space $T_V$, or by pulling back via the Cayley transformation, we can derive the spaces corresponding to 
$\mathrm{VcMO}(\mathbb R)$. These include the closed subspace $M_{Vc}(\mathbb H)$ of Beltrami coefficients, 
the Banach subspaces ${\rm VcMOA}(\mathbb H)$ 
and $A_{Vc}(\mathbb H)$ of holomorphic functions, the closed subset ${\rm SSc}$ of 
quasisymmetric homeomorphisms of $\mathbb R$, and the Teichm\"uller subspace $T_{Vc}$.

For these spaces, the equivalent conditions as in Proposition \ref{Vequivalence} also hold (see Shen and Wei \cite[Theorem 4.1]{SWei}).
Furthermore, the aforementioned defects no longer exist. The following claims are established in the setting of
$\mathbb S$ and $\mathbb D$. Claim (1) is found in Anderson, Becker and Lesley \cite[p.458]{ABL}, (2) is in Wei \cite[Theorem 4.1]{Wei}, 
and (3) is verified using a claim in Garc\'ia-Cuerve and Rubio de Francia \cite[p.474]{GR}.

\begin{proposition}\label{TVC}
$(1)$ The composition operator $C_h$ for $h \in {\rm SQS}$ preserves $\mathrm{VcMO}(\mathbb R)$.
$(2)$ $T_{Vc} \cong {\rm SSc}$ is a closed subgroup of $T_B \cong {\rm SQS}$. In fact, it is a topological subgroup.
$(3)$ $\mathrm{VcMO}(\mathbb R) \subset \mathrm{BMO}^*(\mathbb R)$.
Thus, $F^\mu(\mathbb R)$ is a chord-arc curve for $\mu \in M_{Vc}(\mathbb H)$, and 
$T_{Vc}$ is a proper closed subset of $T_C$.
\end{proposition}

Furthermore, the chord-arc curve $F^\mu(\mathbb R)$ given by $\mu \in M_{Vc}(\mathbb H^+)$
satisfies the following condition on the chord-arc constant $\kappa \geq 1$. See 
Pommerenke \cite[Theorem 2]{Pom78}.

\begin{definition}
A chord-arc curve $\Gamma$ in $\mathbb C$ 
passing through $\infty$ is called {\it asymptotically smooth} 
if the chord-arc constant $\kappa \geq 1$ tends to $1$ uniformly as the spherical distance between 
two points on $\Gamma$ tends to $0$. Namely, setting $|\widetilde{z_1z_2}|$ as
the length of the arc between any two points $z_1, z_2 \in \Gamma$ and $d_{\widehat{\mathbb C}}(z_1,z_2)$ as the spherical distance
on $\widehat{\mathbb C}$,
$$
\lim_{t \to 0} \sup_{d_{\widehat{\mathbb C}}(z_1,z_2) <t} \frac{|\widetilde{z_1z_2}|}{|z_1 - z_2|} = 1.
$$
\end{definition}

\begin{remark}
We have an equivalent definition for chord-arc curves even if we replace the Euclidean distance with
the spherical distance. Moreover, the condition for a Jordan curve $\Gamma$ in the Riemann sphere 
$\widehat{\mathbb C}$ to be a chord-arc curve is invariant under M\"obius transformations. 
See MacManus \cite[p.877]{Mac} and Tukia \cite[Section 7]{Tu}. The above definition for asymptotically smoothness
is the translation of that for bounded Jordan curves in \cite{Pom78} to the unbounded case.
\end{remark}

\begin{proposition}\label{AS}
For $\mu \in M(\mathbb H^+)$, 
$F^\mu(\mathbb R)$ is asymptotically smooth if and only if $\log(F^\mu|_{\mathbb H^-})'$ belongs to
${\rm VcMOA}(\mathbb H^-)$. This is equivalent to the condition
$\mu \in M_{Vc}(\mathbb H^+)$.
\end{proposition}

From this fact and the group structure of $T_{Vc} \cong {\rm SSc}$, we conclude that
for a BMO embedding $\gamma = \gamma([\mu^+],[\mu^-])$ with $([\mu^+],[\mu^-]) \in T_B^+ \times T_B^-$,
its image $\gamma(\mathbb R)$ is asymptotically smooth if and only if
$([\mu^+],[\mu^-]) \in T_{Vc}^+ \times T_{Vc}^-$. 
We refer to this $\gamma$ as an {\it asymptotically smooth embedding}.
In this case, we have $\log \gamma' \in \mathrm{VcMO}(\mathbb R)$. 
Thus, $T_{Vc}^+ \times T_{Vc}^- \subset \widetilde T_C$. 

We now consider the restriction of the biholomorphic map $\Lambda$ to $T_{Vc}^+ \times T_{Vc}^-$.
This yields a biholomorphic homeomorphism
$$
\Lambda:T_{Vc}^+ \times T_{Vc}^- \to \mathrm{VcMO}(\mathbb R)
$$
onto its image, which is a connected open subset of $\mathrm{VcMO}(\mathbb R)$. In addition,
$\Lambda$ maps ${\rm Sym}(T_{Vc}^+ \times T_{Vc}^-)$ onto ${\rm Re}\,{\rm VcMO}(\mathbb R)$
as a real-analytic diffeomorphism. 
Moreover,
we can construct a holomorphic right inverse of $\Lambda$ on some neighborhood of 
the real subspace ${\rm Re}\,{\rm VcMO}(\mathbb R)$, as in the following assertion.
This is a translation of the corresponding results by
WM \cite{WM-6} in the case of $\mathbb S$.

\begin{theorem}
There exists a neighborhood $W$ of ${\rm Re}\,{\rm VcMO}(\mathbb R)$ in ${\rm VcMO}(\mathbb R)$
and a holomorphic map $\Sigma:W \to M_{Vc}(\mathbb H^+) \times M_{Vc}(\mathbb H^-)$ such that
$(\pi^+ \times \pi^-) \circ W$ is a holomorphic right inverse of 
$\Lambda:T_{Vc}^+ \times T_{Vc}^- \to {\rm VcMO}(\mathbb R)$ on $W$.
\end{theorem}

\begin{proof}
This follows from the arguments around \cite[Theorem 6.2]{WM-6}. Since $L^\infty(\mathbb{R})$ is dense in ${\rm VcMO}(\mathbb{R})$,
we can define a holomorphic map $\Sigma$ on some neighborhood $W$ of ${\rm Re}\,{\rm VcMO}(\mathbb{R})$ using the variant of
the Beurling--Ahlfors extension as in Theorem \ref{FKP}. By a proof similar to that of \cite[Proposition 6.3]{WM-6}, we see that 
$$
\Sigma(W) \subset M_{B}(\mathbb{H}^+) \times M_{B}(\mathbb{H}^-); \quad 
\Sigma({\rm Re}\,{\rm VcMO}(\mathbb{R})) \subset M_{Vc}(\mathbb{H}^+) \times M_{Vc}(\mathbb{H}^-),
$$
and in particular, $\Sigma|_{{\rm Re}\,{\rm VcMO}(\mathbb{R})}$ is real-analytic. However, by combining the proof of \cite[Theorem 4.1]{WM-2}
with that of \cite[Lemma 5.4]{WM-6}, we can actually show that $\Sigma(W)$ is contained in $M_{Vc}(\mathbb{H}^+) \times M_{Vc}(\mathbb{H}^-)$.
The results for ${\rm VMO}(\mathbb{R})$ are able to be applied to ${\rm VcMO}(\mathbb{R}) \cong {\rm VMO}(\mathbb{S})$ by lifting
functions on $\mathbb{S}$ via the universal cover $\mathbb{R} \to \mathbb{S}$.
\end{proof}

We investigate the correspondence from the arc-length parametrization to the Riemann mapping parametrization for asymptotically smooth curves. Analogous to the general case, we define
\begin{align}
Z_c &=i\,\mathrm{BMO}(\mathbb R) \cap \Lambda(T_{Vc}^+ \times T_{Vc}^-); \\
Y_c&=\mathrm{BMOA}(\mathbb H^+) \cap \Lambda(T_{Vc}^+ \times T_{Vc}^-) \subset \mathrm{VcMOA}(\mathbb H^+).
\end{align}
The former is a real analytic submanifold, and
the latter is a complex analytic submanifold of $\Lambda(T_{Vc}^+ \times T_{Vc}^-)$. 
Moreover, we set $\widetilde Z_c = \Lambda^{-1}(Z_c)$ and $\widetilde Y_c = \Lambda^{-1}(Y_c)$, 
and consider the map
$$
\rho: T_{Vc}^+ \times T_{Vc}^- \to \{[0]\} \times T_{Vc}^-, \quad ([\mu^+],[\mu^-]) \mapsto ([0], [\mu^-] \ast [\overline{\mu^+}]^{-1}).
$$
Let $\rho_0 = \rho|_{\widetilde Z_c}$. 

By the topological group property of $T_{Vc}$, as established in Proposition \ref{TVC}, we derive the following result in contrast to Theorem \ref{discontinuous}. For Weil--Petersson Teichm\"uller spaces,
the corresponding claim is in WM \cite[Proposition 6.5]{WM-4}.

\begin{theorem}
$\rho_0: \widetilde Z_c \to \widetilde Y_c$ is a homeomorphism.
\end{theorem}

\begin{proof}
It is straightforward to verify that $\rho_0$ is bijective. Since $\rho$ is continuous due to the topological group property of $T_{Vc}$, it follows that $\rho_0$ is also continuous. Hence, it remains to show that $\rho_0^{-1}$ is continuous. To do this, we consider the conjugation of $\rho_0^{-1}$ by $\Lambda$, that is,
$$
\Lambda \circ \rho^{-1}_0 \circ \Lambda^{-1}: Y_c \to Z_c.
$$
For any $\varphi \in Y_c$, we have
$$
\Lambda \circ \rho^{-1}_0 \circ \Lambda^{-1}(\varphi) = C_h^{-1}(i\, \mathrm{Im} \, \varphi),
$$
where $h(x) = \int_0^x \exp(\mathrm{Re} \, \varphi(t)) \, dt$. Since $h \mapsto h^{-1}$ is continuous in $T_{Vc} \cong \mathrm{SSc}$, we conclude that the map
$$
\mathrm{VcMO}(\mathbb R) \times T_{Vc} \to \mathrm{VcMO}(\mathbb R), \quad (\phi, h) \mapsto C_h^{-1}(\phi),
$$
is continuous by Lemma \ref{strong} below. Thus, $\Lambda \circ \rho^{-1}_0 \circ \Lambda^{-1}$ is continuous.
\end{proof}

\begin{lemma}\label{strong}
The map $\mathrm{VcMO}(\mathbb R) \times T_{Vc} \to \mathrm{VcMO}(\mathbb R)$ defined by
$(\phi, h) \mapsto C_h(\phi)$ is continuous. In particular, a sequence of bounded linear operators
$C_{h_n}$ on $\mathrm{VcMO}(\mathbb R)$ converges to $C_h$ strongly as $h_n \to h$ in $T_{Vc}$.
\end{lemma}

\begin{proof}
The affine translation $Q_h(\phi)$ of $\phi \in \mathrm{VcMO}(\mathbb R)$ is given by $C_h(\phi) + \log h'$. Then, we have
$$
\Lambda \circ R_{[\mu]} = Q_h \circ \Lambda
$$
for $h=h(\mu) \in {\rm SSc}$. From this, we deduce
$$
C_h(\phi) = \Lambda \circ R_{[\mu]} \circ \Lambda^{-1}(\phi) - \log h',
$$
which is a continuous function of $(\phi, h)$. The continuity with respect to $h$ follows from the topological group property 
of $T_{Vc}$.
\end{proof}

\end{document}